\documentclass[12pt,reqno]{amsart}

\usepackage{latexsym}
\usepackage{amssymb}
\usepackage{mathrsfs}
\usepackage{amsmath}
\usepackage{fancybox,color}
\usepackage{enumerate}
\usepackage[latin1]{inputenc}

\usepackage[left=2.3cm,top=2.65cm,right=2.3cm]{geometry}

\def\1{\raisebox{2pt}{\rm{$\chi$}}}

\newtheorem{theorem}{Theorem}[section]
\newtheorem{corollary}[theorem]{Corollary}
\newtheorem{lemma}[theorem]{Lemma}
\newtheorem{proposition}[theorem]{Proposition}
\newtheorem{definition}[theorem]{Definition}
\newtheorem{remark}[theorem]{Remark}

\newtheorem{example}[theorem]{Example}

\newcommand{\R}{{\mathbb R}}

\newcommand{\N}{{\mathbb N}}

\DeclareMathOperator{\udima}{\overline{dim}_A}

\DeclareMathOperator{\dima}{dim_A}

\def\1{\raisebox{2pt}{\rm{$\chi$}}}

\def\vint_#1{\mathchoice%
        {\mathop{\kern 0.2em\vrule width 0.6em height 0.69678ex depth -0.58065ex
                \kern -0.8em \intop}\nolimits_{\kern -0.4em#1}}%
        {\mathop{\kern 0.1em\vrule width 0.5em height 0.69678ex depth -0.60387ex
                \kern -0.6em \intop}\nolimits_{#1}}%
        {\mathop{\kern 0.1em\vrule width 0.5em height 0.69678ex
            depth -0.60387ex
                \kern -0.6em \intop}\nolimits_{#1}}%
        {\mathop{\kern 0.1em\vrule width 0.5em height 0.69678ex depth -0.60387ex
                \kern -0.6em \intop}\nolimits_{#1}}}
\def\vintslides_#1{\mathchoice%
        {\mathop{\kern 0.1em\vrule width 0.5em height 0.697ex depth -0.581ex
                \kern -0.6em \intop}\nolimits_{\kern -0.4em#1}}%
        {\mathop{\kern 0.1em\vrule width 0.3em height 0.697ex depth -0.604ex
                \kern -0.4em \intop}\nolimits_{#1}}%
        {\mathop{\kern 0.1em\vrule width 0.3em height 0.697ex depth -0.604ex
                \kern -0.4em \intop}\nolimits_{#1}}%
        {\mathop{\kern 0.1em\vrule width 0.3em height 0.697ex depth -0.604ex
                \kern -0.4em \intop}\nolimits_{#1}}}

\newcommand{\intav}{\vint}
\newcommand{\aveint}[2]{\mathchoice%
        {\mathop{\kern 0.2em\vrule width 0.6em height 0.69678ex depth -0.58065ex
                \kern -0.8em \intop}\nolimits_{\kern -0.45em#1}^{#2}}%
        {\mathop{\kern 0.1em\vrule width 0.5em height 0.69678ex depth -0.60387ex
                \kern -0.6em \intop}\nolimits_{#1}^{#2}}%
        {\mathop{\kern 0.1em\vrule width 0.5em height 0.69678ex depth -0.60387ex
                \kern -0.6em \intop}\nolimits_{#1}^{#2}}%
        {\mathop{\kern 0.1em\vrule width 0.5em height 0.69678ex depth -0.60387ex
                \kern -0.6em \intop}\nolimits_{#1}^{#2}}}

\newcommand{\dist}{\operatorname{dist}}

\title[Beyond local maximal operators]{Beyond local maximal operators}

\author[H.\! Luiro]{Hannes Luiro}   
\address[H.L.]{Department of Mathematics and Statistics, P.O. Box 35, FI-40014 University of Jyv\"askyl\"a, Finland}
\email{hannes.s.luiro@jyu.fi}

\author{Antti V. V\"ah\"akangas}
\address[A.V.V.]{Department of Mathematics and Statistics, P.O. Box 35, FI-40014 University of Jyv\"askyl\"a, Finland}
\email{antti.vahakangas@iki.fi}

\date{\today}

\begin{document}


\begin{abstract}
We obtain (essentially sharp) boundedness results for 
certain generalized local maximal operators between fractional weighted Sobolev spaces and their modifications. 
Concrete boundedness results between 
well known fractional Sobolev spaces 
are derived as consequences of our main result.
We also apply our boundedness results by studying both 
generalized neighbourhood capacities and the
Lebesgue differentiation of fractional weighted Sobolev functions.
\end{abstract}

\maketitle

\markboth{\textsc{H. Luiro  and A. V. V\"ah\"akangas}}
{\textsc{Beyond local maximal operators}}

\section{Introduction}

In this paper we study the boundedness of a centered maximal-type operator
$M_R$ between fractional $A_p$ weighted Sobolev spaces and their $R$-modifications;
the well known fractional Sobolev spaces $W^{s,p}(G)$ 
 for open sets $G\subset\R^n$ are special cases of the aforementioned spaces, see \S\ref{s.weighted_Sobo}.
The operator $M_R$ depends on a given
measurable function $R:G\to \R$ which  satisfies the condition  $0\le R(x)\le \dist(x,\partial G)$ whenever $x\in G$.
Here, and throughout the paper, 
we agree that $\mathrm{dist}(x,\partial G)=\infty$ if $x\in G=\R^n$.
For any $f\in L^1_{\textup{loc}}(G)$ and every $x\in G$ we define
\begin{equation}\label{d.max}
M_R(f)(x) =   \sup_r \intav_{B(x,r)} \lvert f(y)\rvert\,dy\,,
\end{equation}
 where the supremum is taken  over all radii $0\le  r\le   R(x)$
and we have used the notational convention
\begin{equation}\label{e.zero}
\intav_{B(x,0)} \lvert f(y)\rvert \,dy = \lvert f(x)\rvert\,.
\end{equation}
Even though special cases of this maximal-type operator have been studied earlier, cf. below, we are not aware of previous studies in this generality and in connection with Sobolev spaces.
There is  a parallel problem of fixing the appropriate Sobolev spaces where the boundedness 
is to be studied; to illustrate, let us remark that $M_R$ need not preserve the  smoothness of order $0<s\le 1$, unless $R$ is (say) a Lipschitz function.

Our main result shows that fractional $A_p$ weighted Sobolev
spaces and their $R$-modified counterparts, \S\ref{s.weighted_Sobo}, are well-suited for studying the boundedness properties of  $M_R$; this result can be found in \S\ref{s.proof_main}. 
The main result will be applied to the study  of
certain neighbourhood capacities (see \S \ref{s.neighbour}) and
the Lebesgue differentiation of fractional weighted
Sobolev functions (see \S \ref{s.Lebesgue}).
We expect that there are  other
applications in fractional weighted potential theory; indeed,
an operator $M_R$ that is given by an application specific $R$-function provides a flexible tool that can be used to estimate `size' in terms of `smoothness'. This is especially true when combined with 
fractional Sobolev or Hardy inequalities 
\cite{Dyda2,HV2,ihnatsyeva3}.

More specifically, our main
result is Theorem \ref{t.m.bounded}. This  theorem is a
`fractional Sobolev analogue' of  the celebrated Muckenhoupt's theorem which, in turn,
is a boundedness
result for the Hardy--Littlewood maximal operator on the $A_p$ weighted
$L^p$-spaces (for a detailed formulation, we refer to Proposition
\ref{t.ap_weightedM}).
In order to avoid technicalities at this stage, let us
formulate Theorem \ref{t.m.bounded_app_I} that is a consequence of our main result (when applied
with a Muckenhoupt $A_p$ weight that is defined by $\omega(x)=\lvert x\rvert^{\varepsilon-n}$ for the given $0<\varepsilon<np$).


\begin{theorem}\label{t.m.bounded_app_I}
Let $\emptyset\not=G\subset \R^n$ be an open set, $0<\varepsilon,s < 1$ and $1<p < \infty$.
Fix a measurable function $R:G\to \R$ satisfying inequality $0\le R(x)\le \dist(x,\partial G)$ for every $x\in G$.
Then there exists a constant
$C=C(n,p,\varepsilon)>0$ such that inequality
\begin{equation}\label{e.gen_I}
\int_G\int_{G}\frac{\lvert M_R(f)(x)-M_R(f)(y)\rvert^p}{(\lvert x-y\rvert+\lvert R(x)-R(y)\rvert)^{\varepsilon+sp}}\,\frac{dy\,dx}{\lvert x-y\rvert^{n-\varepsilon}}
\le C\int_G\int_G \frac{\lvert f(x)-f(y)\rvert^p}{\lvert x-y\rvert^{n+sp}} \,dy\,dx
\end{equation}
holds for every $f\in L^p(G)$.
\end{theorem}

We  remark that if $R$ is a Lipschitz function, 
e.g., if $R=\dist(\cdot,\partial G)$ in case of a proper open subset $G$ of $\R^n$, then the left-hand side of inequality \eqref{e.gen_I}
is comparable to 
\[
\int_G\int_G \frac{\lvert M_R(f)(x)-M_R(f)(y)\rvert^p}{\rvert x-y\rvert^{n+sp}}\,dy\,dx\,.
\]
In particular, Theorem \ref{t.m.bounded_app_I} generalizes
a recently obtained boundedness result for 
the local Hardy--Littlewood maximal operator $M_{\dist(\cdot,\partial G)}$
on fractional Sobolev spaces $W^{s,p}(G)$, see \cite[Theorem 1.1]{LV2}.
Another interesting case 
is when $R$ is an $\alpha$-H\"older function ($0<\alpha<1$) on a bounded open set $G$ such that $0\le R(x)\le \dist(x,\partial G)$ for each $x\in G$.
Corollary \ref{c.holder} then implies that 
\[M_R:W^{s,p}(G)\to W^{\sigma,p}(G)\,,\qquad 0<\sigma<\alpha s\,,\] 
is a bounded operator
whenever $0<s<1$ and $1<p<\infty$; 
with the aid of a fractional Hardy inequality
we show in Lemma \ref{e.conv} that this result is essentially sharp, in that we cannot  allow $\sigma>\alpha s$ in general (however, we do not know if 
$\sigma=\alpha s$ is allowed).
In particular, our
main result (Theorem \ref{t.m.bounded}) is also essentially sharp in its generality.

We close this introduction with a brief overview on related results for the maximal and local maximal operators.
The maximal operators $M_{\mathcal{B}}$ that are defined by (differentiation) bases $\mathcal{B}$ have been extensively studied, e.g.,
 in connection with differentiability properties of functions, we refer to \cite{MR807149,Tauberian,MR833361,Kin_Lat,LL}. 
 
 Concerning  
the boundedness of maximal operators on the Sobolev-type spaces, 
previous research has mainly focused on the 
Hardy--Littlewood maximal operator $M$ and the local
maximal operator $M_{\dist(\cdot,\partial G)}$ for a given open set $G\subset \R^n$;
see \cite{MR2041705,MR1469106,MR1979008,MR2280193}.
In particular, the boundedness of the local maximal operator on the first order Sobolev spaces 
$W^{1,p}(G)$ is proved by Kinnunen and Lindqvist \cite{MR1650343}. Their main result
states that if $1<p\le \infty$ and $f\in W^{1,p}(G)$, then $M_{\dist(\cdot,\partial G)}(f)\in W^{1,p}(G)$ and
\begin{equation}\label{e.crelle}
\lvert \nabla (M_{\dist(\cdot,\partial G)}(f))(x)\rvert \le 2M_{\dist(\cdot,\partial G)}(\lvert \nabla f\rvert)(x)
\end{equation}
for almost every $x\in G$;  observe that inequality \eqref{e.crelle} and
boundedness of the local maximal operator on $L^p(G)$  yields boundedness
of the local maximal operator on $W^{1,p}(G)$. We will prove
a fractional weighted counterpart of inequality \eqref{e.crelle} in Proposition \ref{p.maximal}.
Korry \cite{MR1951818} studied boundedness of the Hardy--Littlewood maximal operator on
the Triebel--Lizorkin spaces of the (fractional) order smoothness $0<s<1$.
The first author established in \cite{MR2579688} the boundedness and continuity properties of $M_{\dist(\cdot,\partial G)}$ 
on the (non-intrinsically defined) Triebel--Lizorkin
spaces $F^{s}_{pq}(G)$ for $0<s<1$ and $1<p,q<\infty$. 
Boundedness results
for the discrete analogues of maximal operators in metric spaces can be found, e.g., in \cite{Hei_Tuo,Hei_Kin_Tuo,Kin_Lat}.

\subsection*{Acknowledgments}

The authors wish to thank 
Tuomas Hyt\"onen and Juha Lehrb\"ack for inspiring 
discussions.
H.L. was supported by
the Academy of Finland, grant no.\ 259069.

\section{Notation and preliminaries}\label{s.notation}

The open ball centered at $x\in \R^n$ and with radius $r>0$ is  $B(x,r)$.
The Euclidean
distance from $x\in\R^n$ to a set $E$ in $\R^n$ is denoted by $\dist(x,E)$. Here we agree that $\mathrm{dist}(x,\emptyset)=\infty$.
The Euclidean diameter of $E$  is $\mathrm{diam}(E)$.
The characteristic function of a set $E$ is written as $\chi_E$.
The Lebesgue $n$-measure of a  measurable set $E$ is denoted by $\vert E\vert$.
If $0<|E|<\infty$,
the integral average of a function $f\in L^1(E)$ is 
$f_E=\intav_E f\,dx = |E|^{-1} \int_E f\,dx$.
If $G$ is an open set in $\R^n$, then $C_0(G)$ denotes
the space of continuous functions $f$ in $G$ whose support
\[\mathrm{supp}(f)=\overline{\{x\in G\,:\, f(x)\not=0\}}\]
is a compact set contained in $G$; the closure above is taken in $\R^n$.
 If there exists a constant $C>0$ such that $a\le C b$, we  write $a\lesssim b$,
and if $a\lesssim b\lesssim a$ we write $a\simeq b$ and say that $a$ and $b$ are comparable.
 We let $C(\star,\dotsb,\star)$  denote a positive constant which depends on the quantities appearing
in the parentheses only.


Function $\omega\in L^1_{\textup{loc}}(\R^n)$ 
is a weight if  $\omega(x)>0$ for almost every $x\in \R^n$.
Let $1<p<\infty$.  A weight $\omega$ is an $A_p$ weight
if there exists $A>0$ such that, for every cube $Q\subset\R^n$, 
\[
\bigg(\intav_{Q} \omega\,dx \bigg)\bigg(\intav_Q  \omega^{-1/(p-1)} \,dx\bigg)^{p-1} \le A\,.
\]
The infimum over all such constants $A$ is called the $A_p$ constant 
of $\omega$, written as $[\omega]_{A_p}$.
The Hardy--Littlewood maximal function $Mf$ for a  function
$f\in L^1_{\textup{loc}}(\R^n)$  is defined
by
\[
Mf(x) = \sup_{r>0}\intav_{B(x,r)} \lvert f(y)\rvert\,dy\,,\qquad x\in \R^n\,.
\]
Muckenhoupt's theorem is the following well known result,
see \cite[\S IV.2 Theorem 2.8]{MR807149} for a proof and further details.

\begin{proposition}\label{t.ap_weightedM}
Let $1<p<\infty$ and let $\omega$ be an $A_p$ weight. Then there exists a constant $C>0$ such
that 
\[
\int_{\R^n}(M f(x))^p\, \omega(x)\,dx \le C\int_{\R^n}\lvert f(x)\rvert^p\,\omega(x) \,dx
\]
whenever $f$ is a measurable function for which the
integral on the right-hand side is finite.
Moreover, the constant $C$ depends only on $n$, $p$ and the $A_p$ constant of $\omega$.
\end{proposition}


When $A\subset \R^n$ is bounded and $r>0$, we let $N(A,r)$ denote the minimal
number of (open) balls 
of radius $r$ 
and centered at $A$
that are needed to cover the set $A$.  
For any set $E\subset\R^n$, the (upper) Assouad dimension of $E$
is defined by setting 
\begin{align*}
&\udima(E) \\&= \inf\biggl\{\lambda\ge 0 : N(E\cap B(x,R),r)\le
 C_\lambda\biggl(\frac r R\biggr)^{-\lambda}\ \text{ for all }
 x\in E,\ 0<r<R<\mathrm{diam}(E)\biggr\}\,.
\end{align*}
This is the `usual' Assouad dimension
found in the literature, e.g. in \cite{MR1608518}, 
often
 denoted $\dima(E)$. 
If $E\subset\R^n$ is a (sufficiently) regular set, for instance, Ahlfors $d$-regular, then
the upper Assouad dimension of $E$ coincides with its Hausdorff dimension; we refer to \cite{KLV}.

A set $E\subset \R^n$ is 
$\kappa$-porous ($0<\kappa<1$) if 
for each $x \in E$ and every $0<r<\mathrm{diam}(E)$ there exists a point 
$y\in \R^n$ such that $B(y,\kappa r) \subset B(x,r) \setminus E$.
We remark that a set $E\subset \R^n$ is $\kappa$-porous for some $0<\kappa<1$ if and only if $\udima(E)<n$, see \cite[Theorem~5.2]{MR1608518}.

%

\section{Fractional weighted Sobolev spaces}\label{s.weighted_Sobo}

We present the fractional weighted Sobolev seminorms and the associated function spaces that are used throughout this paper. 
Moreover, we consider the density 
of smooth functions in these spaces  by adapting the argument given in \cite{HV}.
Incidentally, density properties
for other fractional weighted Sobolev spaces have recently been studied in \cite{Dipierro,Fiscella}. 
Since our weights are always translation invariant, the density
arguments are quite straightforward and (eventually) based  upon the
continuity of translations in the classical Lebesgue spaces.
Whereas a similar approach is  used in the work \cite{Fiscella},
 a more refined approximation scheme is developed in \cite{Dipierro}
 to handle weights that are not translation invariant.

The fractional weighted Sobolev seminorm $\lvert f \rvert_{W^{s,p,\omega}(G)}$
given in Definition \ref{d.sob} 
has been previously studied, e.g., in connection with fractional weighted Hardy-type inequalities, extension problems and variational problems, we refer to \cite{MR1624754,Dyda_comparability,Fiscella}.

\begin{definition}\label{d.sob}
Let  $s>0$ and $1\le p<\infty$, and let $\omega$ be a weight in $\R^n$ (see \S\ref{s.notation}).
Fix an open set $G\subset\R^n$. Then $W^{s,p,\omega}(G)$ 
is the fractional weighted Sobolev space of 
functions $f\in L^p(G)$ satisfying
$\lVert f\rVert_{W^{s,p,\omega}(G)}^p
= \lVert f\rVert_{L^p(G)}^p+|f|_{W^{s,p,\omega}(G)}^p<\infty$,
where 
\begin{equation}\label{e.semi_weighted}
\lvert f \rvert_{W^{s,p,\omega}(G)} = \bigg( \int_{G}\int_{G}\frac{\lvert f(x)-f(y)\rvert^p}{\lvert x-y\rvert^{s p}}\,\omega(x-y)\,
dy\,dx\,\bigg)^{1/p}
\end{equation}
is the fractional weighted Sobolev seminorm.
\end{definition}

We remark that the global norm
is translation invariant, i.e., for each $f\in W^{s,p,\omega}(\R^n)$ and every $h\in\R^n$ we have
\[
\lVert f(\cdot+h)\rVert_{W^{s,p,\omega}(\R^n)}=\lVert f\rVert_{W^{s,p,\omega}(\R^n)}\,.
\]
Hence, our framework is most likely not the nearest fractional analogue
of the first order $A_p$ weighted Sobolev space that is not generally translation invariant,
see \cite{Kilpelainen,MR1774162}.

There is  an $R$-modification of the seminorm \eqref{e.semi_weighted}
that will also be  relevant to us. Namely, given a measurable function
$R:G\to \R$, we will often encounter the following (often translation invariantless) seminorm
\begin{equation}\label{e.strong}
\bigg(\int_G\int_G\frac{\lvert f(x)-f(y)\rvert^p}{(\lvert x-y\rvert+ \lvert R(x)-R(y)\rvert)^{s p}}\,\omega(x-y)\,
dy\,dx\,\bigg)^{1/p}\,.
\end{equation}
Theorem \ref{t.m.bounded} and the supporting counterexample
given in \S\ref{holder} indicate that if $\varphi\in W^{s,p,\omega}(G)$, then the right
way to measure the smoothness of $f=M_R(\varphi)$ is to use  \eqref{e.strong}.
This quantity can be viewed as a weighted seminorm that measures `variable fractional smoothness' 
of  $f\in L^p(G)$. Indeed, assuming that $R$ is a Lipschitz function on $G$,
the last seminorm \eqref{e.strong} is comparable to $\lvert f\rvert_{W^{s,p,\omega}(G)}$. 
On the other hand, if $R$ oscillates more significantly then 
\begin{equation}\label{e.big}
\lvert x-y\rvert + \lvert R(x)-R(y)\rvert
\end{equation}
can be much larger than $\lvert x-y\rvert$.
We remark that 
\eqref{e.big} is comparable to Euclidean distance between $(x,R(x))$ and $(y,(R(y))$
that belong to the graph $\{(w,R(w))\,:\,w\in G\}\subset \R^{n+1}$.

We will apply the fractional $A_p$ weighted Sobolev spaces and their $R$-modifications. 
Both of these spaces arise
naturally in the proof of
our main result and, moreover, the well known fractional Sobolev spaces are their special cases:

\begin{example}\label{s.usual_frac}
Consider the well known and widely used fractional Sobolev
space 
$W^{s,p}(G)$, whose survey can be found in \cite{MR2944369}. For any given $\varepsilon>0$ this space 
can be represented as $W^{s+\varepsilon/p,p,w}(G)$ when
the weight is given by $w=\lvert \cdot\rvert^{\varepsilon-n}$.
In particular, the fractional Sobolev seminorm corresponding to \eqref{e.semi_weighted} is  independent of $\varepsilon$ and it is given by 
\[
\lvert f\rvert_{W^{s,p}(G)}=\lvert f\rvert_{W^{s+\varepsilon/p,p,w}(G)}=\bigg(\int_G\int_G\frac{\lvert f(x)-f(y)\rvert^p}{\lvert x-y\rvert^{n+s p}}\,
dy\,dx\,\bigg)^{1/p}\,.
\]
We remark that 
 $\lvert \cdot\rvert^{\varepsilon-n}$ is an  $A_p$ weight if, and only if, inequality
 $0<\varepsilon<np$ holds;
we refer to \cite[p.~229, p.~236]{MR869816}.
\end{example}


We turn to density of continuous functions in $W^{s,p,\omega}(\R^n)$;
this
will be needed in \S\ref{s.Lebesgue} when studying Lebesgue differentiation
and quasicontinuous representatives of fractional weighted Sobolev functions.
Our density argument seems to require that
 $\omega$ has a sufficient decay at infinity that is
quantified by inequality \eqref{e.subset} below. 
This decay inequality turns out to be quite natural:
it
 is equivalent to the requirement that
$C^\infty_0(\R^n)$ is a subset of $W^{s,p,\omega}(\R^n)$.
We also remark that when  $sp<n$
inequality \eqref{e.subset} for a given $\rho>0$ fails even for
the $A_p$ weight that is defined by $\omega(x)=1$ for every $x\in\R^n$.

\begin{lemma}\label{l.c_characterization}
Let $0< s <1$ and $1\le p<\infty$, and let $\omega$ be a weight in $\R^n$.
Then $C^\infty_0(\R^n)$ is a subset $W^{s,p,\omega}(\R^n)$
if, and only if,
\begin{equation}\label{e.subset}
\int_{\R^n\setminus B(0,\rho)}\frac{\omega(x)}{\lvert x\rvert^{sp}}\,dx<\infty
\end{equation}
for every $\rho>0$ (or, equivalently, for some $\rho>0$).
\end{lemma}

\begin{proof}
Let us first assume that $C^\infty_0(\R^n)$ is a subset of $W^{s,p,\omega}(\R^n)$. 
Fix $\rho>0$ and $f\in C^\infty_0(\R^n)$
such that $\mathrm{supp}(f)\subset B(0,\rho/2)$ and $f(0)=2$.
Fix $0<\delta<\rho/2$ such that $f(x)\ge 1$ if $\lvert x\rvert<\delta$.
Then
\begin{align*}
\int_{\R^n\setminus B(0,\rho)}\frac{\omega(x)}{\lvert x\rvert^{sp}}\,dx
&= \intav_{B(0,\delta)}
\int_{\R^n\setminus B(0,\rho)}\frac{\omega(x)}{\lvert x\rvert^{sp}}\,dx\,dy\\
&\le \intav_{B(0,\delta)}
\int_{\R^n\setminus B(0,\rho/2)}\frac{1}{\lvert x-y\rvert^{sp}}\,\omega(x-y)\,dx\,dy\\
&\le \frac{1}{\lvert B(0,\delta)\rvert} \int_{\R^n}
\int_{\R^n}\frac{\lvert f(x)-f(y)\rvert^p}{\lvert x-y\rvert^{sp}}\,\omega(x-y)\,dx\,dy=\frac{\lvert f\rvert_{W^{s,p,\omega}(\R^n)}^p}{\lvert B(0,\delta)\rvert}<\infty\,.
\end{align*}
Hence, inequality \eqref{e.subset} holds. 

Conversely,  let us assume that inequality \eqref{e.subset} holds for some $\rho>0$.
Fix $f\in C^\infty_0(\R^n)$ and  choose $R>\rho$ such that $\mathrm{supp}(f)\subset B(0,R)$. 
Suppose that $x$ and $y$ are in the ball $B(0,2R)$, $x\not=y$. Then,
by the mean-value theorem,
\[
\frac{\lvert f(x)-f(y)\rvert^p}{\lvert x-y\rvert^{sp}}\le 
\lVert \nabla f\rVert_{L^\infty(\R^n)}^p \lvert x-y\rvert^{p(1-s)}\le \lVert \nabla f\rVert_{L^\infty(\R^n)}^p (4R)^{p(1-s)}=M\,.
\]
By assumption $\omega$ is  a weight. In particular, it is locally integrable, see \S\ref{s.notation}. Hence,
\begin{equation}\label{e.ens}
\begin{split}
&\int_{B(0,2R)}\int_{B(0,2R)}
\frac{\lvert f(x)-f(y)\rvert^p}{\lvert x-y\rvert^{sp}}\,\omega(x-y)\,dy\,dx
\\&\le M\int_{B(0,2R)}\int_{B(0,2R)} \omega(x-y)\,dy\,dx
\le M\lvert B(0,2R)\rvert \int_{B(0,4R)} \omega(z)\,dz < \infty\,.
\end{split}
\end{equation}
Furthermore, by the fact that $R>\rho$ and inequality \eqref{e.subset},
\begin{equation}\label{e.cross}
\begin{split}
&\int_{\R^n\setminus B(0,2R)}\int_{B(0,2R)}
\frac{\lvert f(x)-f(y)\rvert^p}{\lvert x-y\rvert^{sp}}\,\omega(x-y)\,dy\,dx\\
&\le \int_{\R^n\setminus B(0,2R)}\int_{B(0,R)}
\frac{\lvert f(y)\rvert^p}{\lvert x-y\rvert^{sp}}\,\omega(x-y)\,dy\,dx
\le \lVert f\rVert_{L^p(\R^n)}^p  \int_{\R^n\setminus B(0,\rho)}
\frac{\omega(z)}{\lvert z\rvert^{sp}}\,dz  <\infty\,.
\end{split}
\end{equation}
A similar computation shows that
\begin{equation}\label{e.second_cross}
\int_{B(0,2R)}\int_{\R^n\setminus B(0,2R)}
\frac{\lvert f(x)-f(y)\rvert^p}{\lvert x-y\rvert^{sp}}\,\omega(x-y)\,dy\,dx
<\infty\,.
\end{equation}
Inequalities \eqref{e.ens}--\eqref{e.second_cross}
and the fact $\mathrm{supp}(f)\subset B(0,R)$ yield that
$f\in W^{s,p,\omega}(\R^n)$.
\end{proof}

Next we focus on weights $\omega$ satisfying 
$C^\infty_0(\R^n)\subset W^{s,p,\omega}(\R^n)$.
Under this restriction it is now straightforward to show that continuous functions are
dense in $W^{s,p,\omega}(\R^n)$.
For this purpose, we let $\varphi\in C^\infty_0(B(0,1))$ be a non-negative bump
function such that
 $\int_{\R^n}\varphi(x)\,dx = 1$.
For $j\in\N$ and $x\in\R^n$, we write $\varphi_j(x)=2^{jn}\varphi(2^j x)$.
Recall that
\[
f*\varphi_j(x)=\int_{\R^n}f(x-z)\varphi_j(z)\,dz
\]
defines a smooth function in $\R^n$ 
and $\lim_{j\to\infty} \lVert f-f*\varphi_j\rVert_{L^p(\R^n)}=0$
whenever $f\in L^p(\R^n)$ and $1\le p< \infty$, see e.g. \cite{MR924157}.

\begin{lemma}\label{l.approx}
Let $0< s <1$ and $1\le p<\infty$, and let $\omega$ be a weight in $\R^n$
such that $C^\infty_0(\R^n)$ is a subset $W^{s,p,\omega}(\R^n)$.
Then for every
$f\in W^{s,p,\omega}(\R^n)$  we have
\begin{equation}\label{e.convolv}
\lVert f-f\ast \varphi_j\rVert_{W^{ s ,p,\omega}(\R^n)}\xrightarrow{j\to\infty} 0\,.
\end{equation}
In particular, the set $C^\infty(\R^n)\cap W^{s,p,\omega}(\R^n)$ is dense in $W^{s,p,\omega}(\R^n)$. 
\end{lemma}

\begin{proof}
The basic ideas for the proof are from \cite{HV}. 
Since the convolutions $f\ast \varphi_j$ converge to $f$ in $L^p(\R^n)$ when $j\to\infty$, it suffices
to show that 
$\lvert f-f*\varphi_j\rvert_{W^{s,p,\omega}(\R^n)}\to 0$ when
$j\to\infty$.
Fix $\varepsilon>0$. We write
\begin{align*}
\lvert f-f\ast \varphi_j\rvert_{W^{ s ,p,\omega}(\R^n)}^p
= \int_{\R^n} \int_{\R^n} \frac{\lvert f(x)-f\ast \varphi_j(x) -f(y)+f\ast \varphi_j(y)\rvert^p}{\lvert x-y\rvert^{s  p}}\,\omega(x-y)\,dy\,dx\,.
\end{align*}
 Since $\lvert f\rvert_{W^{s,p,\omega}(\R^n)}<\infty$, we may  apply the monotone convergence theorem in $\R^n\times \R^n$ in order to obtain a number
$\rho=\rho(\varepsilon,f,\omega)>0$ such that
\begin{equation}\label{e.eps}
\int_{\R^n} \int_{B(x,\rho)} \frac{\lvert f(x)-f(y)\rvert^p}{\lvert x-y\rvert^{s  p}}\,\omega(x-y)\,dy\,dx <\varepsilon\,.
\end{equation}
Now, for any $j\in\N$
\begin{align*}
\int_{\R^n} \int_{B(x,\rho)}&\frac{\lvert f\ast \varphi_j(x)-f\ast \varphi_j(y)\rvert^p}{\lvert x-y\rvert^{s  p}}\,\omega(x-y)\,dy\,dx \\&\lesssim  2^{-jn(p-1)} \int_{\R^n} \varphi_j(z)^p \int_{\R^n} \int_{B(x,\rho)} \frac{\lvert f(x-z)-f(y-z)\rvert^p}{\lvert x-y\rvert^{s  p}}\,
\omega(x-y)\,dy\,dx\,dz\\
&= 2^{-jn(p-1)}\int_{\R^n} \varphi_j(z)^p \int_{{\R^n}}\int_{B(x,\rho)}\frac{\lvert f(x)-f(y)\rvert^p}{\lvert x-y\rvert^{s  p}}\,\omega(x-y)\,dy\,dx\,dz\,.
\end{align*}
Hence, we obtain that
\begin{equation}\label{e.sec}
\begin{split}
&\int_{\R^n} \int_{B(x,\rho)} \frac{\lvert f\ast \varphi_j(x)-f\ast \varphi_j(y)\rvert^p}{\lvert x-y\rvert^{s  p}}\,\omega(x-y)\,dy\,dx
\\&\qquad\qquad\qquad\qquad\lesssim \int_{\R^n}\int_{B(x,\rho)}\frac{\lvert f(x)-f(y)\rvert^p}{\lvert x-y\rvert^{s  p}}\,\omega(x-y)\,dy\,dx< \varepsilon\,.
\end{split}
\end{equation}
From \eqref{e.eps} and \eqref{e.sec} it follows that
\begin{equation}\label{e.step1}
\sup_{j\in\N} \int_{\R^n} \int_{B(x,\rho)} \frac{\lvert f(x)-f\ast \varphi_j(x) -f(y)+f\ast \varphi_j(y)\rvert^p}{\lvert x-y\rvert^{s  p}}\,\omega(x-y)\,dy\,dx \lesssim \varepsilon\,.
\end{equation}
On the other hand, since $C^\infty_0(\R^n)\subset W^{s,p,\omega}(\R^n)$ by assumptions,  
Lemma \ref{l.c_characterization} yields
\[
\int_{\R^n\setminus B(0,\rho)}\frac{\omega(x)}{\lvert x\rvert^{sp}}\,dx < \infty\,.
\]
Moreover, by assumptions, we have $f\in L^p(\R^n)$ and therefore
\begin{equation}\label{e.step2}
\begin{split}
&\int_{\R^n} \int_{{\R^n}\setminus B(x,\rho)} \frac{\lvert f(x)-f\ast \varphi_j(x) -f(y)+f\ast \varphi_j(y)\rvert^p}{\lvert x-y\rvert^{s  p}}\,\omega(x-y)\,dy\,dx \\
&\qquad \lesssim \bigg(\int_{\R^n\setminus B(0,\rho)}
\frac{\omega(x)}{\lvert x\rvert^{sp}}
 \,dx \bigg) \int_{\R^n} \lvert f(x)-f\ast \varphi_j(x)\rvert^p\,dx\xrightarrow{j\to\infty} 0\,;
\end{split}
\end{equation}
here we again used the fact that $f*\varphi_j$ converges to $f$ in $L^p(\R^n)$ when $j\to\infty$.
By combining the estimates \eqref{e.step1} and \eqref{e.step2}, we find that
 $\lvert f-f*\varphi_j\rvert_{W^{s,p,\omega}(\R^n)}\to 0$ when $j\to \infty$.
\end{proof}

\section{A boundedness result for $M_R$}\label{s.proof_main}

We formulate and prove our main result, i.e., Theorem \ref{t.m.bounded},
that
provides a boundedness result for the maximal operator
$M_R$ (see  \eqref{d.max}) from a fractional $A_p$ weighted Sobolev space
to its $R$-modification, both of which  are defined in \S\ref{s.weighted_Sobo}.

The main result is akin to the Muckenhoupt's theorem, i.e., Proposition \ref{t.ap_weightedM}, in that
both sides of inequality \eqref{e.max_bdd} incorporate an $A_p$ weight. 
Another interesting aspect is how
the left-hand side of inequality  \eqref{e.max_bdd}
depends on the given $R$-function; from the viewpoint
of applications, such a dependence is 
both flexible and straightforward to work with.
Moreover, as we will see in \S \ref{holder}, the $R$-dependence  is 
essentially the best possible in this generality.

Recall our notational convention $\mathrm{dist}(x,\partial G)=\infty$ if $x\in G=\R^n$.

\begin{theorem}\label{t.m.bounded}
Assume that  $\emptyset\not=G\subset \R^n$ is an open set, $0\le s \le 2$, and $1<p < \infty$. 
Fix a measurable function $R:G\to \R$ satisfying inequality $0\le R(x)\le \dist(x,\partial G)$ for every $x\in G$.
Then, 
if $\omega$ is an $A_p$ weight in $\R^n$, 
there exists a constant
$C>0$ such that inequality
\begin{equation}\label{e.max_bdd}
\begin{split}
\int_G \int_G \frac{\lvert M_R(f)(x)-M_R(f)(y)\rvert^p}{(\lvert x-y\rvert+\lvert R(x)-R(y)\rvert)^{sp}}&\,\omega(x-y)\,dy\,dx
 \\& \le C \int_G \int_G \frac{ \lvert f(x)-f(y)\rvert^p}{\lvert x-y\rvert^{sp}}\,\omega(x-y)\,dy\,dx
\end{split}
\end{equation}
holds for every $f\in L^p(G)$. The constant $C$ depends
only on $n$, $p$ and the $A_p$ constant of $\omega$.
\end{theorem}

This result is a far-reaching extension of  \cite[Theorem 1.1]{LV2}
whose proof, in turn, applies ideas from  \cite[Theorem 3.2]{MR2579688}. 
Here delicate modifications are required in the proofs due to  the $A_p$ weight and the $R$-function. In the sequel, we
follow outline of the proof in \cite{LV2}; in particular, we  repeat many  details therein without
further notice. 

The proof of Theorem \ref{t.m.bounded}
will be completed at the end of this section.
The main technical tool is a  pointwise inequality that is given in Proposition \ref{p.maximal}.
Moreover, some implications of the Muckenhoupt's theorem are also needed, see Proposition \ref{t.Mapp}. In order to state the latter proposition, we first need some preparations.

Let us fix $i,j\in \{0,1\}$. 
For a measurable function  $F$ on $\R^{2n}$ we write
\begin{equation}\label{e.lower_dim}
M_{ij}(F)(x,y) = \sup_{r>0}\intav_{B(0,r)} \lvert F(x+iz,y+jz)\rvert \,dz
\end{equation}
whenever the right-hand side is well-defined, i.e., for almost every $(x,y)\in \R^{2n}$
by Fubini's theorem.
Observe that $M_{00}(F)=\lvert F\rvert$.
The measurability
of $M_{ij}(F)$  can be checked by first noting that the supremum in \eqref{e.lower_dim} can
be restricted to the rational numbers $r>0$
and then adapting
the proof of \cite[Theorem 8.14]{MR924157} with each $r$ separately.

By applying Fubini's theorem in appropriate coordinates
and $L^p(\R^n)$-boundedness of the Hardy--Littlewood maximal operator $f\mapsto Mf$ we find that $M_{ij}=(F\mapsto M_{ij}(F))$ is 
a bounded operator on $L^p(\R^{2n})$ whenever $1<p<\infty$.
Furthermore, we need the following $A_p$ weighted norm inequalities
that eventually rely on Muckenhoupt's theorem.

\begin{proposition}\label{t.Mapp}
Let $1<p<\infty$. Then, if $\omega$ is an $A_p$ weight in $\R^n$, there
exists a constant $C=C(n,p,[\omega]_{A_p})>0$ such that
\begin{equation}\label{e.right}
\int_{\R^n}\int_{\R^n} \big(M_{kl}(F)(x,y)\big)^p \,\omega(x-y)\,dy\,dx \le C \int_{\R^n}\int_{\R^n} \lvert F(x,y)\rvert^p \,\omega(x-y)\,dy\,dx
\end{equation}
whenever $F$ is a measurable function in $\R^{2n}$ and $k,l\in \{0,1\}$ are such that $kl=0$.
\end{proposition}

\begin{proof}
We focus on the case $(k,l)=(0,1)$; the case $(k,l)=(1,0)$ is analogous, and the claim is trivial when $k=0=l$.
Let us consider a measurable function $F$ on $\R^{2n}$ for which the double integral on the right-hand side of \eqref{e.right} is finite. 
By dilation and translation invariance of the $A_p$-condition, we find
that the
function
$y\mapsto \omega_{x}(y):= \omega(x-y)$, $x\in\R^n$,
belongs to $A_p$, and its $A_p$ constant coincides with 
$[\omega]_{A_p}$.
Hence, by Proposition \ref{t.ap_weightedM},
\begin{align*}
&\int_{\R^n}\int_{\R^n} \big(M_{01}(F)(x,y)\big)^p \,\omega(x-y)\,dy\,dx
=\int_{\R^n}\int_{\R^n} \big(M(F(x,\cdot))(y)\big)^p\,\omega_{x}(y) \,dy\,dx\\
&\lesssim \int_{\R^n}\int_{\R^n}\lvert F(x,y)\rvert^p\,\omega_{x}(y) \,dy\,dx
=\int_{\R^n}\int_{\R^n}\lvert F(x,y)\rvert^p\,\omega(x-y)\,dy\,dx\,,
\end{align*}
and the proof is complete.
\end{proof}

\begin{remark}
The directional maximal operators $M_{ij}$ are 
dominated by the so-called strong maximal operator, whose 
certain weighted
norm inequalities can be found in \cite[\S IV.6]{MR807149}.
\end{remark}

The following proposition gives a certain extension of  inequality \eqref{e.crelle}.
But first let us introduce further convenient notation that is used in the
remaining part of this section.
We write $\omega_0(x,y)=\omega(x-y)^{1/p}$ and
$\omega_1(x,y)=\omega(y-x)^{1/p}$ if $x,y\in\R^n$ and $\omega$
is an $A_p$ weight in $\R^n$.
For $f\in L^p(G)$ we denote
\[
S_R(f)(x,y) = S_{R,G,s}(f)(x,y)=\frac{ \chi_G(x)\chi_G(y) \lvert f(x)-f(y)\rvert}{(\lvert x-y\rvert + \lvert  R(x)-R(y)\rvert)^{s}}
\]
for almost every $(x,y)\in \R^{2n}$; we also abbreviate $S(f)=S_0(f)$.

\begin{proposition}\label{p.maximal}
Assume that $\emptyset\not=G\subset\R^n$ is an open set, $0\le s\le 2$, and $1<p<\infty$. 
Let $\omega$ be an $A_p$  weight in $\R^n$ and let
  $R:G\to \R$ be a measurable function such that 
  \[0\le R(x) \le\dist(x,\partial G)\] for every $x\in G$. 
  Then there exists a constant $C=C(n)>0$ such that,
for almost every $(x,y)\in\R^{2n}$,  inequality
\begin{equation}\label{e.dom}
\begin{split}
\omega(x-y)^{1/p} &S_R(  M_R(f))(x,y)
\\&\le C\sum_{\substack{ {i,j,k,l,m\in \{0,1\}} \\ {kl=0} }}\big( M_{ij}( \omega_m M_{kl}(S f))(x,y) +M_{ij}( \omega_m M_{kl}(S f))(y,x)\big)
\end{split}
\end{equation}
holds whenever  $f\in L^p(G)$ satisfies the condition 
$\{ \omega_0 S f, \omega_1 S f\}\subset L^p(\R^{2n})$.
\end{proposition}

\begin{proof}
By replacing the function $f$ with $\lvert f\rvert$ we may assume that
$f\ge 0$. Since $f\in L^p(G)$ and, hence, $M_R(f)\in L^p(G)$ we may restrict
ourselves to $(x,y)\in G\times G$ for which both 
$x$ and $y$ are Lebesgue points of $f$ and both $M_R(f)(x)$ and $M_R(f)(y)$ are finite. 
By symmetry,  and changing the weight to $\widetilde\omega$ if necessary, we may further assume that  $M_R(f)(x)>M_R(f)(y)$. 
These reductions allow us to
find \[0\le r(x)\le R(x)\quad \text{ and } \quad 0\le r(y)\le R(y)\] such that 
\begin{align*}
S_R(M_R(f))(x,y)&=\frac{\lvert M_R(f)(x)-M_R(f)(y)\rvert}{(\lvert x-y\rvert + \lvert  R(x)-R(y)\rvert)^{s}}
=\frac{\lvert \intav_{B(x,r(x))}f\,-\intav_{B(y,r(y))}f\,\rvert}{(\lvert x-y\rvert + \lvert  R(x)-R(y)\rvert)^{s}}\,.
\end{align*}
Moreover, since $M_R(f)(x) > M_R(f)(y)$, we find that inequality
\begin{equation}\label{e.lahtee}
S_R(M_R(f))(x,y) \leq\frac{\lvert \intav_{B(x,r(x))}f\,-\intav_{B(y,r_2)}f\,\rvert}{(\lvert x-y\rvert + \lvert  R(x)-R(y)\rvert)^{s}}
\end{equation}
is valid for any number $0\le r_2\le R(y)$;
this number will
be chosen in a convenient manner in the two case studies below.

{\bf{Case $r(x)\le \lvert x-y\rvert + \lvert R(x)-R(y)\rvert$}.}
Let us denote $r_1=r(x)$ and choose
\begin{equation}\label{e.case}
r_2=0\,.
\end{equation}
If $r_1=0$, then we get from \eqref{e.lahtee} and \eqref{e.case}---and our notational convention \eqref{e.zero}---that 
\[
\omega_0(x,y) S_R(M_R(f))(x,y)\leq \omega_0(x,y) S(f)(x,y)\,.
\]
Suppose then that $r_1>0$. Now, by \eqref{e.lahtee},
\begin{align*}
&\omega_0(x,y) S_R(M_R(f))(x,y)\\&\le 
\frac{\omega_0(x,y)}{(\lvert x-y\rvert + \lvert  R(x)-R(y)\rvert)^{s}}\bigg{|} 
  \intav_{B(x,r_1)}f(z)\,dz\,-\intav_{B(y,r_2)}f(z)\,dz\,\bigg{|}  \\
  &=\frac{\omega_0(x,y)}{(\lvert x-y\rvert + \lvert  R(x)-R(y)\rvert)^{s}} \bigg{|}\intav_{B(x,r_1)} \big( f(z)-f(y) \big)\,dz\,\,\bigg{|}\\
& \lesssim \omega_0(x,y) \intav_{B(0,r_1)} \frac{\chi_G(x+z)\chi_G(y)\lvert f(x+z)-f(y)\rvert}{\lvert x+z-y\rvert^{s}}\,dz\le 
\omega_0(x,y)\,M_{10}(S f)(x,y)\,.
\end{align*}
We have shown that, in the case under consideration,
\begin{align*}
\omega_0(x,y)S_R(M_R(f))(x,y)
\lesssim \omega_0(x,y) S (f)(x,y) + \omega_0(x,y)\,M_{10}(S f)(x,y)\,.
\end{align*}
It is clear that inequality \eqref{e.dom} follows; recall that $M_{00}$ is the identity
operator when restricted to non-negative functions.

{\bf{Case $r(x)>  \lvert x-y\rvert + \lvert R(x)-R(y)\rvert$}.} 
Let us denote $r_1=r(x)>0$ and choose 
\[
0<r_2 = r(x) - \lvert x-y\rvert - \lvert R(x)-R(y)\rvert\le R(y)\,.
\] 
We then have  
\begin{align*}
&\bigg{|}\intav_{B(x,r_1)}f(z)\,dz\,-\intav_{B(y,r_2)}f(z)\,dz\,\bigg{|}
=\bigg{|}\intav_{B(0,r_1)} \bigg(f(x+z)-f(y+\frac{r_2}{r_1}z)\bigg)\,dz\,\bigg{|}\\
&= \bigg{|}\intav_{B(0,r_1)}\bigg(f(x+z)-\intav_{B(y+\frac{r_2}{r_1}z,\lvert x-y\rvert+\lvert R(x)-R(y)\rvert)\cap G}f(a)\,da\\ &\qquad\qquad\qquad \qquad  +\intav_{B(y+\frac{r_2}{r_1}z,\lvert x-y\rvert+\lvert R(x)-R(y)\rvert)\cap G}f(a)\,da
-f(y+\frac{r_2}{r_1}z)\bigg)\,dz\,\bigg{|}\\ 
&\leq E_1+E_2\,,
\end{align*}
where we have written
\begin{align*}
E_1 &= \intav_{B(0,r_1)}\bigg(\intav_{B(y+\frac{r_2}{r_1}z,\lvert x-y\rvert+\lvert R(x)-R(y)\rvert)\cap G}|f(x+z)-f(a)|\,da\,\bigg)\,dz\,,\\
E_2&= \intav_{B(0,r_1)}\bigg(\intav_{B(y+\frac{r_2}{r_1}z,\lvert x-y\rvert+\lvert R(x)-R(y)\rvert)\cap G}|f(y+\frac{r_2}{r_1}z)-f(a)|\,da\,\bigg)\,dz\,.
\end{align*}
We  estimate both of these terms separately,
but first we need certain auxiliary estimates.

Recall that $r_2=r_1-\lvert x-y\rvert -\lvert R(x)-R(y)\rvert$. 
Hence, for every $z\in B(0,r_1)$,
\begin{align*}
\lvert y+\frac{r_2}{r_1}z-(x+z)\rvert&=\lvert y-x+\frac{(r_2-r_1)}{r_1} z\rvert
 \leq 2\lvert x-y\rvert
+\lvert R(x)-R(y)\rvert \,.
\end{align*}
This, in turn, implies that 
\begin{equation}\label{e.inclusion}
B(y+\frac{r_2}{r_1}z,\lvert x-y\rvert+\lvert R(x)-R(y)\rvert)\subset B(x+z,3\lvert x-y\rvert
+2\lvert R(x)-R(y)\rvert) 
\end{equation}
if $z\in B(0,r_1)$.
Since $r_1> \lvert x-y\rvert + \rvert R(x)-R(y)\rvert$  and $\{y+\frac{r_2}{r_1}z, x+z\}\subset B(x,r_1)\subset G$ 
if $\lvert z\rvert< r_1$,
we obtain the two equivalences
\begin{equation}\label{e.equivalence}
\begin{split}
\lvert B(y+\frac{r_2}{r_1}z,\lvert x-y\rvert+\lvert R(x)-R(y)\rvert)\cap G\rvert &\simeq 
(\lvert x-y\rvert
+\lvert R(x)-R(y)\rvert)^n\\
& \simeq  \lvert B(x+z,3\lvert x-y\rvert 
+2\lvert R(x)-R(y)\rvert)\cap G\rvert 
\end{split}
\end{equation}
for every $z\in B(0,r_1)$. Here the implied constants depend only on $n$.

{\bf An estimate for $E_1$}.
The inclusion \eqref{e.inclusion} and equivalences \eqref{e.equivalence} show that, in the definition of $E_1$, we can replace the set over
which the inner integral its taken  by the set
\[B(x+z,3\lvert x-y\rvert+2\lvert R(x)-R(y)\rvert)\cap G\] and, at the same time,  control the error term while integrating on average. That is,
\begin{align*}
&E_1 \lesssim \intav_{B(0,r_1)}\bigg(\intav_{B(x+z,3\lvert x-y\rvert
+2\lvert R(x)-R(y)\rvert)\cap G}  |f(x+z)-f(a)|\,da\,\bigg)\,dz\,.
\end{align*}
By observing that $x+z$ and $a$ in the last double integral belong to  $G$, and using \eqref{e.equivalence} again, we can continue as follows:
\begin{align*}
&\frac{E_1 \omega_0(x,y)}{(\lvert x-y\rvert + \lvert  R(x)-R(y)\rvert)^{s}}
\\&\lesssim
\intav_{B(0,r_1)}\bigg(   \omega_0(x,y) \intav_{B(x+z,3\lvert x-y\rvert
+2\lvert R(x)-R(y)\rvert)} 
\frac{ \chi_G(x+z)\chi_G(a)\lvert f(x+z)-f(a)\rvert}{\lvert x+z-a\rvert^{s}}\,da\bigg)dz\\
&\lesssim\intav_{B(0,r_1)}\bigg(  \omega_0(x,y) \intav_{B(y+z,4\lvert x-y\rvert
+2\lvert R(x)-R(y)\rvert)}
 S(f)(x+z,a)\,da\,\bigg)\,dz\,.
\end{align*}
Since $\omega_0(x,y)=\omega_0(x+z,y+z)$, we may apply the maximal operators defined in \S \ref{s.notation} in order to find that
\begin{equation}\label{e.e_1}
\begin{split}
&\frac{E_1 \omega_0(x,y)}{
(\lvert x-y\rvert + \lvert  R(x)-R(y)\rvert)^{s}}\\&\lesssim \intav_{B(0,r_1)} \omega_0(x+z,y+z)M_{01}(Sf)(x+z,y+z)\,dz\leq M_{11}(\omega_0 M_{01}(Sf))(x,y)\,.
\end{split}
\end{equation}

{\bf An estimate for $E_2$.} 
We use the inclusion $y+\frac{r_2}{r_1}z\in G$ for all  $z\in B(0,r_1)$ and then apply 
the first equivalence in \eqref{e.equivalence} to obtain
\begin{align*}
E_2 &= \intav_{B(0,r_1)}\bigg(\intav_{B(y+\frac{r_2}{r_1}z,\lvert x-y\rvert+\lvert R(x)-R(y)\rvert)\cap G} 
\chi_G(y+\frac{r_2}{r_1}z)\chi_G(a)\lvert f(y+\frac{r_2}{r_1}z)-f(a)\rvert \,da\,\bigg)\,dz\\
&\lesssim \intav_{B(0,r_1)}\bigg(\intav_{B(y+\frac{r_2}{r_1}z, \lvert x-y\rvert+\lvert R(x)-R(y)\rvert)} 
\chi_G(y+\frac{r_2}{r_1}z)\chi_G(a)\lvert f(y+\frac{r_2}{r_1}z)-f(a)\rvert \,da\,\bigg)\,dz\,.
\end{align*}
Hence, a change of variables yields
\begin{align*}
&\frac{E_2 \omega_0(x,y)}{(\lvert x-y\rvert + \lvert  R(x)-R(y)\rvert)^{s}}
\\
&\lesssim \intav_{B(0,r_2)} \omega_0(x,y)\bigg(\intav_{B(y+z,\lvert x-y\rvert+\lvert R(x)-R(y)\rvert)}\frac{\chi_G(y+z)\chi_G(a)\lvert f(y+z)-f(a)\rvert}{\lvert y+z-a\rvert^{s}}\,d
a\,\bigg)\,dz\\
&\lesssim \intav_{B(0,r_2)} \omega_0(x,y)\bigg(\intav_{B(x+z,2\lvert x-y\rvert+\lvert R(x)-R(y)\rvert)}  
S(f)(y+z,a)\,da\,\bigg)\,dz\,.
\end{align*}
Let us observe that $\omega_0(x,y)=\omega_1(y+z,x+z)$. Hence,
by applying operators $M_{01}$ and $M_{11}$ from \S \ref{s.notation}, we can proceed as follows
\begin{equation}\label{e.e_2}
\begin{split}
&\frac{E_2 \omega_0(x,y)}{(\lvert x-y\rvert + \lvert  R(x)-R(y)\rvert)^{s}}
\\&\lesssim\intav_{B(0,r_2)} \omega_1(y+z,x+z) M_{01}(Sf)(y+z,x+z)\,dz\leq M_{11}(\omega_1 M_{01}(Sf))(y,x)\,.
\end{split}
\end{equation}

By combining the estimates \eqref{e.e_1} and \eqref{e.e_2}, we obtain  that
\begin{align*}
\omega_0(x,y)S_R(M_R(f))(x,y)&\le\frac{(E_1+E_2)\omega_0(x,y)}{(\lvert x-y\rvert + \lvert  R(x)-R(y)\rvert)^{s}}  \\
&\lesssim M_{11}(\omega_0 M_{01}( S f))(x,y)+M_{11}(\omega_1 M_{01}(S f))(y,x)\,,
\end{align*}
where the implied constant depends only on $n$. As a consequence, inequality \eqref{e.dom} follows also in  the second case 
$r(x)>\lvert x-y\rvert + \lvert R(x)-R(y)\rvert$ that is now under
our consideration.
\end{proof}

We are finally ready to prove Theorem \ref{t.m.bounded}.
\begin{proof}[Proof of Theorem \ref{t.m.bounded}]
Let  $f\in L^p(G)$.
We may assume that the double integral on the right hand side of  \eqref{e.max_bdd} is finite, and therefore
$\omega_m S f \in L^p(\R^{2n})$ if $m\in \{0,1\}$. 
Observe that $\omega_1(x,y)^p=\widetilde{\omega}(x-y)$, where $\widetilde{\omega}(z)=\omega(- z)$ is also an $A_p$ weight such that $[\widetilde{\omega}]_{A_p}=[\omega]_{A_p}$.
Hence, inequality \eqref{e.max_bdd} is a consequence of Proposition \ref{p.maximal}, the boundedness
of  operators $M_{ij}$ on $L^p(\R^{2n})$, and Proposition \ref{t.Mapp} applied with
the two $A_p$ weights $\omega$ and $\widetilde{\omega}$.
\end{proof}

\section{Powers of distance as weights}\label{holder}

In this section we  apply Theorem \ref{t.m.bounded} with
$\omega=\dist(\cdot,E)^{\varepsilon-n}$,
where the set $E\subset\R^n$ and $\varepsilon>0$ are chosen such that $\omega$ is an $A_p$ weight in $\R^n$; we refer to Theorem \ref{c.dist}.
The important special  case $E=\{0\}$ and $\omega=\lvert \cdot \rvert^{\varepsilon-n}$ yields
 boundedness 
results for the operators $f\mapsto M_R(f)$ between  fractional Sobolev
spaces. 
These results with Lipschitz and
 H\"older functions $R$ are formulated in Corollaries 
 \ref{c.lipschitz} and \ref{c.holder}, respectively.
The sharpness  of  Corollary \ref{c.holder}
in terms of the H\"older exponent is considered
 in Lemma \ref{e.conv}. Furthermore, this lemma shows that Theorem \ref{t.m.bounded}, i.e., our main result, is essentially
 sharp in its generality.

 The following proposition can be found in  \cite{ihnatsyeva4}
(see also \cite{LV} and \cite[Lemma 2.2]{MR1118940}).
The straightforward proof relies on a 
characterization of the Assouad dimension in terms of the so-called Aikawa dimension, we refer to \cite{LT}.

\begin{proposition}\label{t.a_infty}
Let $E\subset\R^n$ be a (non-empty) closed set and let
 $\omega=\dist(\cdot,E)^{\varepsilon-n}$ for a fixed $\varepsilon>0$.
Then the following statements are true.
\begin{itemize}
\item[(A)] If $\udima(E)<\varepsilon\le n$, then $\omega$ is an $A_p$ weight in $\R^n$ for all $1 < p <\infty$.
\item[(B)] If $\varepsilon>n$ and $1<p<\infty$ are such that
\[
\udima(E)< n - \frac{\varepsilon-n}{p-1}\,, 
\]
then $\omega$ is an $A_p$ weight in $\R^n$.
\end{itemize}
\end{proposition}

The following result illustrates the flexibility of our main result, and it is an
immediate consequence of
Theorem \ref{t.m.bounded} and Proposition \ref{t.a_infty}.

\begin{theorem}\label{c.dist}
Assume that  $\emptyset\not=G\subset \R^n$ is an open set,  
$0\le  s \le  2$, and $1<p < \infty$.
Let $\varepsilon>0$ and $E\not=\emptyset$ be a closed set in $\R^n$  such that
\[
\udima(E) < \varepsilon < n + (n-\udima(E))(p-1)\,.
\] 
Fix a measurable function $R:G\to \R$ satisfying inequality $0\le R(x)\le \dist(x,\partial G)$ for every $x\in G$.
Then there exists a constant
$C=C(n,p,\varepsilon,E)>0$ such that inequality
\begin{equation*}
\begin{split}
\int_G\int_{G}  \frac{\lvert M_R(f)(x)-M_R(f)(y)\rvert^p}{(\lvert x-y\rvert+\lvert R(x)-R(y)\rvert)^{sp}}\,&\,\frac{dy\,dx}{\mathrm{dist}(x-y,E)^{n-\varepsilon}}
\\&\le C\int_G\int_{G}  \frac{ \lvert f(x)-f(y)\rvert^p}{\lvert x-y\rvert^{sp}}\,\frac{dy\,dx}{\mathrm{dist}(x-y,E)^{n-\varepsilon}}
\end{split}
\end{equation*}
holds for every $f\in L^p(G)$.
\end{theorem}

Next we turn to an important special case, where $E=\{0\}$ and $\omega=\dist(\cdot,E)^{\varepsilon-n}=\lvert \cdot \rvert^{\varepsilon-n}$.
The following convenient result is a 
reformulation of
 Theorem \ref{t.m.bounded_app_I}; for
 the definition of the seminorm appearing in the right-hand side
 of \eqref{e.gen}, we refer to Example \ref{s.usual_frac}.

\begin{proposition}\label{t.m.bounded_app}
Let $\emptyset\not=G\subset \R^n$ be an open set, $0<\varepsilon,s < 1$ and $1<p < \infty$.
Fix a measurable function $R:G\to \R$ satisfying inequality $0\le R(x)\le \dist(x,\partial G)$ for every $x\in G$.
Then there exists a constant
$C=C(n,p,\varepsilon)>0$ such that inequality
\begin{equation}\label{e.gen}
\int_G\int_{G}\frac{\lvert M_R(f)(x)-M_R(f)(y)\rvert^p}{(\lvert x-y\rvert+\lvert R(x)-R(y)\rvert)^{\varepsilon+sp}}\,\frac{dy\,dx}{\lvert x-y\rvert^{n-\varepsilon}}
\le C \lvert f\rvert_{W^{s,p}(G)}^p
\end{equation}
holds for every $f\in L^p(G)$. 
\end{proposition}

\begin{proof}
Since $0<\varepsilon < 1$, we find that the function $\lvert x\rvert^{\varepsilon-n}$
is an $A_p$ weight; see \cite[p.~236]{MR869816} or
Proposition \ref{t.a_infty}(A).
Moreover, the $A_p$ constant of this weight depends only on $n$, $p$ and $\varepsilon$.
Observe also that $\varepsilon/p+s < 2$.
Hence, inequality \eqref{e.gen} follows from Theorem \ref{t.m.bounded}.
\end{proof}

\begin{remark}
Observe that Proposition \ref{t.m.bounded_app} is related to the 
 case $E=\{0\}$ of Theorem \ref{c.dist}. Indeed, we have
that $\udima(\{0\})=0$.
\end{remark}

 The following boundedness result, which applies for Lipschitz $R$-functions,
 is a corollary of Proposition~\ref{t.m.bounded_app}.
Let us fix $L\ge 0$ and recall that  $R$ is an $L$-Lipschitz function on $G$ if 
 $\lvert R(x)-R(y)\rvert\le L\lvert x-y\rvert$ whenever $x,y\in G$.

\begin{corollary}\label{c.lipschitz}
Let $\emptyset\not=G\subset \R^n$ be an open set,  $0<\varepsilon,s < 1$ and $1<p < \infty$.
Fix  $L\ge 0$ and an $L$-Lipschitz function $R:G\to \R$ satisfying inequality $0\le R(x)\le \dist(x,\partial G)$ for every $x\in G$.
Then there exists a constant
$C=C(n,p,\varepsilon)>0$ such that inequality
\[
\lvert M_R(f)\rvert_{W^{s,p}(G)}\le C(1+L)^{\varepsilon/p+s} \lvert f\rvert_{W^{s,p}(G)}
\]
holds for every function $f\in L^p(G)$.
\end{corollary}

The case of H\"older functions $R$ is
addressed in Corollary \ref{c.holder} below.
Let us recall that 
a function $R$ is $\alpha$-H\"older on $G$ (for a given $0<\alpha < 1$) if there exists $L\ge 0$ such that
inequality $\lvert R(x)-R(y)\rvert \le L\lvert x-y\rvert^\alpha$ holds whenever $x,y\in G$.

\begin{corollary}\label{c.holder}
Let $\emptyset\not=G\subset \R^n$ be a bounded open set, $0<s,\alpha < 1$ and $1<p < \infty$. 
Fix an $\alpha$-H\"older function $R$ on $G$ such that
$0\le R(x)\le \dist(x,\partial G)$ for every $x\in G$.
Then, 
if $0<\sigma<\alpha s$, there exists a constant
$C=C(\sigma,s,\alpha,n,p,L,\mathrm{diam}(G))>0$ such that inequality
\begin{equation}\label{e.ws}
\lvert M_R(f)\rvert_{W^{\sigma,p}(G)} \le  C \lvert f\rvert_{W^{s,p}(G)}
\end{equation}
holds for every $f\in L^p(G)$.
\end{corollary}

We omit the  proof of  Corollary \ref{c.holder} that 
is  quite a straightforward 
but tedious reduction to
Proposition \ref{t.m.bounded_app}; it is worthwhile to emphasize
that the open set $G$ is assumed
to be bounded. Hence, the case when $\sigma$ is close to $\alpha s$ is 
a difficult one to establish. 

It is unknown to the authors, 
whether inequality \eqref{e.ws} holds also when $\sigma=\alpha s$ is the endpoint.
However, the following Lemma \ref{e.conv} shows that  Corollary \ref{c.holder} is essentially sharp, in that
we cannot allow $\sigma>\alpha s$ in general.


\begin{lemma}\label{e.conv}
 Fix $\alpha = 1/M$ for any number
 $M\in \{2,3,\ldots\}$. Let $0<s<1$ and $1<p<\infty$ be such that 
$\alpha sp\ge 1$. Then there exists a bounded  open set $G$ in $\R$ and an $\alpha$-H\"older function $R:G\to \R$
which satisfies the inequality $0\le R(x)\le \dist(x,\partial G)$ whenever  $x\in G$
and which has the following property:
for any given $\sigma\in (\alpha s,1)$
there does not exist a constant $C>0$ such that 
\[\lvert M_R(f)\rvert_{W^{\sigma,p}(G)}\le C\lvert f\rvert_{W^{s,p}(G)}\]
for all functions $f\in L^p(G)$.
\end{lemma}

\begin{proof}
Let us fix $\alpha s < \sigma < 1$ and
first sketch the proof that relies on 
a fractional $(\sigma,p)$-Hardy inequality: there exists a constant  $C(p,\sigma)>0$ such that
\begin{equation}\label{e.hardy_1}
\int_I\int_I \frac{\lvert f(x)-f(y)\rvert^p}{\lvert x-y\rvert^{1+\sigma p}}\,dy\,dx\ge C(p,\sigma)\int_I \frac{\lvert f(x)\rvert^p}{\dist(x,\partial I)^{\sigma p}}\,dx
\end{equation}
whenever $f\in L^p(I)$ is  compactly supported in an open interval $I\subset \R$,  \cite[Corollary 2.7]{MR2659764}. Actually, this corollary is  formulated only for
$C^\infty_0(I)$ functions, and therefore approximation by such functions
is required. This can easily be done by a straightforward combination of 
Lemma \ref{l.approx} and
\cite[Lemma 4.4]{HV}; we omit the details.

We  take $G=(-8,9)$ and construct $R$ and test functions $\psi_N$ ($N\ge 1$) that are supported in an interval $I_N\subset G$  such
that $M_R(\psi_N)$ has 
a compactly supported bump in many dyadic subintervals $I_{N,j}\subset I_N$. 
Hence, the fractional $(\sigma,p)$-Hardy inequality
applies to the restriction of $M_R(\psi_N)$ in each of the subintervals.
The resulting estimates, when combined with an upper bound for $\lvert \psi_N\rvert_{W^{s,p}(G)}$, will yield that
$\lvert M_R(\psi_N)\rvert_{W^{\sigma,p}(G)}/\lvert \psi_N\rvert_{W^{s,p}(G)}\to \infty$
as $N\to \infty$.

Let us now turn to the details.
We set 
\[E=G\setminus \bigg( \bigcup_{N\in \N} \bigcup_{j=1}^{2^{(M-1)N}} I_{N,j}\bigg)\,, \]
where
\[
I_{N,j}=\big( 2^{-N} + (j-1)2^{-MN},2^{-N} + j2^{-MN}\big)\,,\qquad j=1,\ldots,2^{(M-1)N}\,.\]
Define an $\alpha$-H\"older function $R= 2^{2\alpha+1}\dist(\cdot,E)^\alpha$.
It is now straightforward to check that inequality
$0\le R(x)\le \dist(x,\partial G)$ holds for every $x\in G$.

Let $\psi\in C^\infty_0((0,1))$ be such that
$\int_{\R} \lvert \psi\rvert\,dx  = 4$.
Fix $N\in\N$  and  write $\psi_N(x)=\psi(2^N x)$ if
$x\in \R$.  Now $\psi_N$ is supported in $(0,2^{-N})$ and,
by a change of variables, we find that
\begin{equation}\label{e.psi}
\lvert \psi_N\rvert_{W^{s,p}(G)}^p \le  \lvert \psi_N\rvert_{W^{s,p}(\R)}^p=
2^{N(sp-1)}\lvert \psi\rvert_{W^{s,p}(\R)}^p\,.
\end{equation}
Next we turn to establishing a lower bound for $\lvert M_R(\psi_N)\rvert_{W^{\sigma,p}(G)}$. Let us fix $j=1,\ldots,2^{(M-1)N}$ and $x\in \widehat{I}_{N,j}=2^{-1} I_{N,j}$.
Since $(0,2^{-N})\subset B(x,2^{-N+1})$ and $2^{-N+1}\le R(x)$, we obtain  that
 \begin{equation}\label{e.max}
M_R (\psi_N)(x) \ge \vint_{B(x,2^{-N+1})} \lvert \psi_N(y)\rvert \,dy
\ge
2^{-N+2}/2^{-N+2} = 1\,,\qquad x\in \widehat{I}_{N,j}\,.
\end{equation}
Moreover, 
the restriction $M_R(\psi_N)\lvert_{I_{N,j}}\in L^p(I_{N,j})$ has a compact support in $I_{N,j}$. Hence,
by the fractional $(\sigma,p)$-Hardy inequality \eqref{e.hardy_1} followed by inequality \eqref{e.max},
\begin{align*}
\lvert M_R(\psi_N)\rvert_{W^{\sigma,p}(G)}^p &\ge \sum_{j=1}^{2^{(M-1)N}}\int_{I_{N,j}}\int_{I_{N,j}}
\frac{\lvert M_R (\psi_N)(x)-M_R (\psi_N)(y)\rvert^p}{\lvert x-y\rvert^{1+\sigma p}}\,dy\,dx\\
&\ge C(p,\sigma)\sum_{j=1}^{2^{(M-1)N}} \int_{I_{N,j}} \frac{\lvert M_R(\psi_N)(x)\rvert^p}{\mathrm{dist}(x,\partial I_{N,j})^{\sigma p}}\,dx\\
&\ge C(p,\sigma)\sum_{j=1}^{2^{(M-1)N}} \lvert \widehat{I}_{N,j}\rvert^{1-\sigma p}
\ge C(p,\sigma)\sum_{j=1}^{2^{(M-1)N}} (2^{-MN})^{1-\sigma p} = C(p,\sigma)2^{N(\sigma pM-1)}\,.
\end{align*}
By combining the estimate above with \eqref{e.psi}, we obtain that
\[
\frac{\lvert M_R(\psi_N)\rvert_{W^{\sigma,p}(G)}^p}{\lvert \psi_N\rvert_{W^{s,p}(G)}^p}\ge 
C(p,\sigma) \lvert \psi\rvert_{W^{s,p}(\R)}^{-p}2^{pN(\sigma M-s)}=C(p,\sigma)\lvert \psi\rvert_{W^{s,p}(\R)}^{-p}2^{pN(\sigma /\alpha-s)}\,.
\]
Since $\sigma>\alpha s$,  the lower bound above tends to infinity as $N\to \infty$.
\end{proof}

\section{Sobolev capacity and Lebesgue differentiation}\label{s.Lebesgue}

We apply our main result by studying the Lebesgue differentiation of a Sobolev function $f\in W^{s,p,\omega}(\R^n)$ 
outside a set of zero Sobolev capacity, see Definition \ref{e.sob_cap}.
The  outline of our treatment is based on the work \cite{Kin_Lat} of Kinnunen--Latvala, who obtain
Lebesgue point results for (first order) Sobolev functions on metric spaces.
We adapt their treatment to the
present setting when $\omega$ is an $A_p$ weight
that is subject to the condition $C^\infty_0(\R^n)\subset W^{s,p,\omega}(\R^n)$.
Hence, the key ingredients for the proof of our Theorem \ref{t.Lebesgue} are:
the density property
of continuous functions  (Lemma \ref{l.approx}) and
the boundedness of (an appropriate) maximal operator, both in $W^{s,p,\omega}(\R^n)$; the boundedness property follows from our main result, i.e., Theorem \ref{t.m.bounded}.

\begin{definition}\label{e.sob_cap}
Suppose that $0<s<1$ and $1\le p<\infty$. Let $\omega$ be a weight in $\R^n$.
For a set $E\subset \R^n$ we define its Sobolev capacity
\[
C_{s,p,\omega}(E)=\inf_{\varphi\in\mathcal{A}(E)} \lVert \varphi\rVert_{W^{s,p,\omega}(\R^n)}^p\,,
\]
where the infimum is taken over all admissible functions
\[
\mathcal{A}(E) = \{\varphi\in W^{s,p,\omega}(\R^n)\,:\,\varphi\ge 1\text{ in an open
set containing }E\}\,.
\]
If $\mathcal{A}(E)=\emptyset$, we set $C_{s,p,\omega}(E)=\infty$.
\end{definition}

The unweighted fractional Sobolev capacity (corresponding to  $W^{s,p}(\R^n)=W^{s+\varepsilon/p,p,\omega}(\R^n)$ with $\omega=\lvert \cdot \rvert^{\varepsilon-n}$)
is well known and extensively studied, see, e.g.,
 \cite{MR2839008,MR1953497,Warma}.
Let us
remark that $W^{s,p}(\R^n)$ 
coincides with the Besov space $B^{s}_{p,p}(\R^n)$ and their norms are comparable, 
if $1<p<\infty$ and $0<s<1$; we refer to 
\cite[pp. 6--7]{MR1163193}.

We prove the following result that is concerned with the Lebesgue differentiation 
and a quasicontinuous representative $f^*$ of a function 
$f$ in $W^{s,p,\omega}(\R^n)$. 

\begin{theorem}\label{t.Lebesgue}
Suppose that  $0< s <1$ and $1< p<\infty$. Let $\omega$ be an $A_p$ weight in $\R^n$
such that $C^\infty_0(\R^n)$ is a subset of $W^{s,p,\omega}(\R^n)$.
Then, for every $f\in W^{s,p,\omega}(\R^n)$, there is a $G_\delta$-set $E\subset \R^n$ such
that $C_{s,p,\omega}(E)=0$ and the limit
\begin{equation}\label{e.cont}
\lim_{r\to 0_+}\intav_{B(x,r)} f(y)\,dy = f^*(x)
\end{equation}
exists for every $x\in \R^n\setminus E$. 
Moreover, for every
$\varepsilon>0$, there exists an open set $U\subset \R^n$ such
that 
$C_{s,p,\omega}(U)<\varepsilon$ and $f^*|_{\R^n\setminus U}$ is
well-defined and continuous on $\R^n\setminus U$.
\end{theorem}

An analogue of Theorem \ref{t.Lebesgue} for the first order $A_p$ weighted Sobolev spaces
is known, see \cite{Kilpelainen,MR1774162}.
The unweighted case
$W^{s,p}(\R^n)=W^{s+\varepsilon/p,p,\omega}(\R^n)$ with $\omega=\lvert \cdot \rvert^{\varepsilon-n}$ of Theorem~\ref{t.Lebesgue}
is also known, we refer to \cite{MR1006450} or \cite[\S 6]{MR1411441} when $p=2$. 
The local aspects of quasicontinuity (in the unweighted case) have been studied in
 \cite[Theorem 3.7]{Warma}; however, the Lebesgue differentiation is not explicitly considered therein.

If all singletons have a positive Sobolev capacity, then 
 $f^*:\R^n\to \R$ is continuous. This is a corollary of Theorem \ref{t.Lebesgue} 
and the translation invariance of  $\lVert \cdot\rVert_{W^{s,p,\omega}(\R^n)}$.

\begin{corollary}
Let $0< s <1$ and $1< p<\infty$. Let $\omega$ be an $A_p$ weight in $\R^n$
such that $C^\infty_0(\R^n)$ is a subset of $W^{s,p,\omega}(\R^n)$
and
\[
C_{s,p,\omega}(\{x\}) >0
\]
for every $x\in\R^n$ (or, equivalently, for some $x\in\R^n$). Then every function $f\in W^{s,p,\omega}(\R^n)$ has a
continuous representative. That is, the function $f^*:\R^n\to\R$  
defined by \eqref{e.cont} is continuous and satisfies $f=f^*$ pointwise almost everywhere in $\R^n$.
\end{corollary}

The proof of Theorem \ref{t.Lebesgue} is given in the end of this section;
first we state and prove several auxiliary results.
A key result among these is the following  capacitary weak type estimate, which
is  a counterpart of \cite[Lemma 4.4]{Kin_Lat}.
For every $f\in L^p(\R^n)$, we define
\[
\widehat{M} f(x) = \sup_{0<r\le 1}\intav_{B(x,r)}\lvert f(y)\rvert\,dy\,,\qquad x\in\R^n\,.
\]
Write $R(x)=1$ whenever $x\in\R^n$. Then
$\widehat{M}f(x)=M_{R}(f)(x)$ in the Lebesgue points $x\in\R^n$ 
of $\lvert f\rvert$, that is, almost everywhere. Moreover, we clearly have that
$\widehat{M}f\le Mf$, where $M$ is
the Hardy--Littlewood maximal operator.

\begin{lemma}\label{l.maximal_testing}
Let $0< s <1$ and $1< p<\infty$, and let $\omega$ be an $A_p$ weight in $\R^n$.
Suppose that $f\in W^{s,p,\omega}(\R^n)$. Then, for every  $\lambda>0$, we have
\[
C_{s,p,\omega}(\{x\in\R^n\,:\, \widehat{M} f(x)>\lambda\})\le C\lambda^{-p}\lVert f\rVert_{W^{s,p,\omega}(\R^n)}^p\,,
\]
where $C=C(n,p,[\omega]_{A_p})$.
\end{lemma}

\begin{proof}
Fix $f\in W^{s,p,\omega}(\R^n)$ and $\lambda>0$.
If $0<r\le 1$, then the function
\[
x\mapsto \intav_{B(x,r)}\lvert f(y)\rvert\,dy
\]
is continuous in $\R^n$ by the
dominated convergence theorem. As a consequence, the function
$\widehat{M}f(x)$
is lower semicontinuous in $\R^n$. Hence, 
$E_\lambda = \{x\in \R^n \,:\,\widehat{M}f(x)>\lambda\}$
is an open set in $\R^n$ and $\lambda^{-1} \widehat{M}f(x)\ge 1$ holds
if $x\in E_\lambda$.
 Theorem \ref{t.m.bounded} 
and the boundedness of the Hardy--Littlewood maximal operator in $L^p(\R^n)$
imply that 
\[
C_{s,p,\omega}(E_\lambda)\le \lVert \lambda^{-1} \widehat{M} f\rVert_{W^{s,p,\omega}(\R^n)}^p \le  C(n,p,[\omega]_{A_p}) \lambda^{-p}\lVert f\rVert_{W^{s,p\omega}(\R^n)}^p\,.
\]
This concludes the proof.
\end{proof}

The following lemma is an adaptation of  \cite[Theorem 3.2]{Kin_Mar}.

\begin{lemma}\label{l.outer_measure}
Let $0< s <1$ and $1\le p<\infty$, and let $\omega$ be a weight in $\R^n$.
Then $C_{s,p,\omega}$ is an outer measure on $\R^n$. 
\end{lemma}

\begin{proof}
By definition, $C_{s,p,\omega}(\emptyset)=0$ and
$C_{s,p,\omega}$ is monotone, that is,
$C_{s,p,\omega}(E)\le C_{s,p,\omega}(F)$ whenever $E\subset F$.
To prove subadditivity, we suppose that $E_i$, $1=1,2,\ldots$, are subsets of $\R^n$. 
We need to establish the inequality 
\begin{equation}\label{e.sub}
C_{s,p,\omega}\bigg(\bigcup_{i=1}^\infty E_i\bigg)\le \sum_{i=1}^\infty C_{s,p,\omega}(E_i)\,.
\end{equation}
We may clearly assume that $\sum_{i=1}^\infty C_{s,p,\omega}(E_i)<\infty$. Let 
us fix $\varepsilon>0$.
For every $i=1,2,\ldots$ it holds that $\mathcal{A}(E_i)\not=\emptyset$ and, therefore,
we can choose $\varphi_i\in \mathcal{A}(E_i)$ such that
\begin{equation}\label{e.ch}
\lVert \varphi_i\rVert_{W^{s,p,\omega}(\R^n)}^p \le C_{s,p,\omega}(E_i)+\varepsilon 2^{-i}\,.
\end{equation}
By replacing each function $\varphi_i$ with $\min\{1,\lvert \varphi_i\rvert\}$ we may assume
that $0\le \varphi_i\le 1$ everywhere and $\varphi_i=1$ in an open
set containing $E_i$.

Define $\varphi = \sup_i  \varphi_i$.
Then $\varphi= 1$ in an open set containing $\cup_{i=1}^\infty E_i$.
Let us fix $x,y\in\R^n$. If $\varphi(x)\ge \varphi(y)$ then,
for every $\delta>0$, there is $j=j(\delta,x)\in\N$ such that
\begin{align*}
\lvert \varphi(x)-\varphi(y)\rvert &=\varphi(x)-\varphi(y) \le \varphi_j(x)+\delta  - \varphi(y)
\\&\le \varphi_j(x)+\delta - \varphi_j(y)
\le \delta+\lvert \varphi_j(x)-\varphi_j(y)\rvert \le \delta+\bigg(\sum_{i=1}^\infty \lvert \varphi_i(x)-\varphi_i(y)\rvert^p\bigg)^{1/p}\,.
\end{align*}
Taking $\delta\to 0$ we obtain that
\begin{equation}\label{e.bound}
\lvert \varphi(x)-\varphi(y)\rvert\le \bigg(\sum_{i=1}^\infty \lvert \varphi_i(x)-\varphi_i(y)\rvert^p\bigg)^{1/p}\,.
\end{equation}
By repeating the previous argument if
$\varphi(x)< \varphi(y)$, with the obvious changes, we find that  inequality \eqref{e.bound} holds
for all $x,y\in\R^n$. Therefore, 
\begin{equation}\label{e.sobo}
\lvert \varphi\rvert_{W^{s,p,\omega}(\R^n)}^p
\le 
\sum_{i=1}^\infty\lvert \varphi_i\rvert_{W^{s,p,\omega}(\R^n)}^p\,.
\end{equation}
Clearly, we also have
\begin{equation}\label{e.lp}
\lVert \varphi\rVert_{L^p(\R^n)}^p \le \sum_{i=1}^\infty \lVert \varphi_i\rVert_{L^p(\R^n)}^p\,.
\end{equation}
By first combining inequalities  \eqref{e.sobo} and \eqref{e.lp}, 
and then using \eqref{e.ch}, we obtain
that
\begin{align*}
\lVert \varphi\rVert_{W^{s,p,\omega}(\R^n)}^p 
\le \sum_{i=1}^\infty \lVert \varphi_i\rVert_{W^{s,p,\omega}(\R^n)}^p
\le \sum_{i=1}^\infty \big(C_{s,p,\omega}(E_i)+ \varepsilon 2^{-i}\big)= 
\varepsilon + \sum_{i=1}^\infty  C_{s,p,\omega}(E_i)\,.
\end{align*}
Hence,  $C_{s,p,\omega}(\cup_{i=1}^\infty E_i)\le \varepsilon + \sum_{i=1}^\infty  C_{s,p,\omega}(E_i)$. 
Inequality \eqref{e.sub} follows by taking $\varepsilon\to 0$.
\end{proof}

\begin{proof}[Proof of Theorem \ref{t.Lebesgue}]
We follow the proof of \cite[Theorem 4.5]{Kin_Lat} very closely; the details
are provided for completeness and convenience of the reader.

Fix a function $f\in W^{s,p,\omega}(\R^n)$ and $i\in \N$. By the assumptions and Lemma \ref{l.approx}, we may choose $f_i\in C^\infty(\R^n)\cap W^{s,p,\omega}(\R^n)$ such that
\[
\lVert f-f_i\rVert_{W^{s,p,\omega}(\R^n)}^p \le 2^{-i(p+1)}\,.
\]
Denote $A_i = \{x\in\R^n\,:\, \widehat{M}(f-f_i)(x) > 2^{-i}\}$. 
Lemma \ref{l.maximal_testing} implies that
\begin{equation}\label{e.sml}
C_{s,p,\omega}(A_i)\le C 2^{ip} \lVert f-f_i\rVert_{W^{s,p,\omega}(\R^n)}^p
\le C2^{-i}\,,
\end{equation}
where $C=C(n,p,[\omega]_{A_p})$. Now (say) for every $x\in\R^n$ and $0<r\le 1$,
\begin{align*}
\lvert f_i(x) - f_{B(x,r)} \rvert \le \intav_{B(x,r)} \lvert f_i(x)-f_i(y)\rvert\,dy
+ \intav_{B(x,r)} \lvert f(y)-f_i(y)\rvert\,dy\,,
\end{align*}
which (by the continuity of $f_i$) implies that
\[
\limsup_{r\to 0} \lvert f_i(x)-f_{B(x,r)}\rvert \le\widehat{M}(f-f_i)(x)\le 2^{-i}\,,\qquad
x\in \R^n\setminus A_i\,.
\]

Let us fix $k\in\N$ and write $B_k=\cup_{i=k}^\infty A_i$. An application of  both subadditivity
of the Sobolev capacity, given by Lemma \ref{l.outer_measure}, 
and inequality \eqref{e.sml} yields
\begin{equation}\label{e.ze}
C_{s,p,\omega}(B_k)\le \sum_{i=k}^\infty C_{s,p,\omega}(A_i)
\le C\sum_{i=k}^\infty 2^{-i}\,.
\end{equation}
If $x\in \R^n\setminus B_k$ and $i,j\ge k$, then
\begin{align*}
\lvert f_i(x)-f_j(x)\rvert \le \limsup_{r\to 0} \lvert f_i(x)-f_{B(x,r)}\rvert
+\limsup_{r\to 0} \lvert f_j(x)-f_{B(x,r)}\rvert \le 2^{-i}+2^{-j}\,.
\end{align*}
It follows that $(f_i)_{i\in\N}$ converges uniformly in $\R^n\setminus B_k$ to a continuous 
function $g_k$ on $\R^n\setminus B_k$. Moreover, if $x\in \R^n\setminus B_k$
and $i\ge k$,  we have
\begin{align*}
\limsup_{r\to 0} \lvert g_k(x)-f_{B(x,r)}\rvert &\le \lvert g_k(x) - f_i(x)\rvert + \limsup_{r\to 0}
\lvert f_i(x)-f_{B(x,r)}\rvert\\&\le \lvert g_k(x) - f_i(x)\rvert + 2^{-i}\,.
\end{align*}
Hence, by taking $i\to \infty$, we obtain that
\[
g_k(x) = \lim_{r\to 0} \intav_{B(x,r)}f(y)\,dy =f^*(x)
\]
for every $x\in \R^n\setminus B_k$. 

Let us define $E=\cap_{k=1}^\infty B_k$. Then,
by monotonicity of the Sobolev capacity and \eqref{e.ze}, 
\[
C_{s,p,\omega}(E)\le \lim_{k\to \infty} C_{s,p,\omega}(B_k)=0
\]
and the limit
\[
\lim_{r\to 0}\intav_{B(x,r)} f(y)\,dy =f^*(x)
\]
does exist for every $x\in \R^n\setminus E$.
Finally, we fix $\varepsilon>0$ and choose $k$ large enough
so that $C_{s,p,\omega}(B_k)<\varepsilon$. By arguing as in 
the proof of Lemma \ref{l.maximal_testing}, we find 
that  $B_k=\cup_{i=k}^\infty A_i$ is a
union of open sets in $\R^n$, hence $B_k$ is open. 
Since $f^*=g_k$ in $\R^n\setminus B_k$,
we find that $f^*|_{\R^n\setminus B_k}$ is continuous on $\R^n\setminus B_k$.
Accordingly, we  can choose $U=B_k$.
\end{proof}

\section{Comparison of neighbourhood capacities}\label{s.neighbour}

As another application of our main result,  Theorem \ref{t.m.bounded}, we prove a capacitary comparison inequality that is formulated as Theorem \ref{t.cap_comp} below;  this inequality extends
the work \cite{L} of Lehrb\"ack.
To briefly explain our inequality, let us fix  a compact $\kappa$-porous set $E$ (see \S\ref{s.notation}) that is contained
in a bounded open set $G\subset \R^n$. 
We write \[E_{t,R}=\{x\in G\,:\, R(x)<t\}\,,\qquad t>0\,,\] 
where $R:G\to \R$ is a
continuous function such that $R= 0$ on $E$. Hence, 
the set $E_{t,R}$ is
an open neighborhood of $E$ in $G$.
We  focus on small values of $t>0$ and the underlying open set $G$
serves for the purpose of an `ambient space'. In particular, the structure of $E_{t,R}$ near to the boundary  $\partial G$
will be irrelevant
to us. 
Our `frame of reference' in comparison is the $t$-neighbourhood
\[E_t=E_{t,\mathrm{dist}(\cdot,E)}=\{x\in G\,:\, \mathrm{dist}(x,E)< t\}\,,\qquad t>0\,.\]
Namely, our capacitary comparison
inequality  is that
%
\[
\mathrm{cap}_{s,p,\omega,R} ( E, E_t\cap E_{4t/\kappa,R},G) \le C\, 
\mathrm{cap}_{s,p,\omega}(E,E_{t},G)
\]
for all small $t>0$ with a constant $C=\kappa^{-np}\,C(n,p,[\omega]_{A_p})>0$;
this is inequality \eqref{e.refss}.
Observe that an $R$-modified 
relative capacity is used in the left-hand side above, whereas an relative $(s,p,\omega)$-capacity is used in the right-hand side; these
are defined as follows.


\begin{definition}\label{d.new_cap}
Let $0<s<1$ and $1<p<\infty$, and let  $\omega$ be a  weight in $\R^n$.
Suppose that  $G\subset\R^n$ is an open set and 
$R:G\to \R$ is a measurable
function.
Let $E\subset \R^n$ be a compact set that is contained
in an open set $H\subset G$.  Then we write
\[
\mathrm{cap}_{s,p,\omega,R}(E,H,G)=\inf_{\varphi} \int_G \int_G \frac{\lvert \varphi(x)-\varphi(y)\rvert^p}{(\lvert x-y\rvert + \lvert R(x)-R(y)\rvert)^{sp}}\,\omega(x-y)\,dy\,dx\,,
\]
where the infimum is taken over all real-valued functions $\varphi\in C_0(G)$
 such that $\varphi(x)\ge 1$ for every $x\in E$ and $\varphi(x)=0$ for every
 $x\in G\setminus H$. If $R(x)=0$ for  every $x\in G$, then we abbreviate
$\mathrm{cap}_{s,p,\omega,R}(E,H,G)=\mathrm{cap}_{s,p,\omega}(E,H,G)$.
 \end{definition}


Let us still clarify the previous definition;

\begin{remark}
Observe that  $\mathrm{cap}_{s,p,\omega,R}(E,H,G)$
need not coincide with $\mathrm{cap}_{s,p,\omega,R}(E,H,H)$; cf.
\cite[p. 598]{Maz} and \cite{Fiscella}. This non-locality 
contributes to our 
heavy notation,
involving several parameters. However, the number of parameters reduces
when we look at special cases:
In the light
of Example \ref{s.usual_frac}, the `relative $(s,p)$-capacity'
\[
\mathrm{cap}_{s,p}(E,H,G)=\mathrm{cap}_{s+\varepsilon/p,p,\lvert\cdot\rvert^{\varepsilon-n}}(E,H,G)=\inf_{\varphi} \int_G \int_G \frac{\lvert \varphi(x)-\varphi(y)\rvert^p}{\lvert x-y\rvert^{n+sp}}\,dy\,dx
\]
is obtained as a special case of our general framework. This relative
$(s,p)$-capacity, in turn, generalizes the following `fractional $(s,p)$-capacity'
\[
\mathrm{cap}_{s,p}(E,G) = \mathrm{cap}_{s,p}(E,G,G)\,,\qquad E\subset G\text{ compact}\,.
\]
These fractional $(s,p)$-capacities have recently found applications, e.g., in connection with
the fractional Hardy inequalities, we refer to \cite{Dyda3,LV2}. 
\end{remark}

The following is our capacity comparison result.

\begin{theorem}\label{t.cap_comp}
Fix $0<s<1$, $1<p<\infty$, and  an $A_p$ weight  $\omega$ in $\R^n$.
Suppose that $E\not=\emptyset$ is a compact $\kappa$-porous set, contained in a bounded open set $G\subset \R^n$,
and  that $R:G\to \R$ is a continuous function
satisfying  both
$R(x)=0$ for every $x\in E$ and
$0\le R(x)\le \mathrm{dist}(x,\partial G)$ for every 
$x\in G$. 
Then, if $0<t< \kappa\, \mathrm{diam}(E)/4$
is such that $\overline{E_t}\subset G$, we have
\begin{equation}\label{e.refss}
\mathrm{cap}_{s,p,\omega,R} ( E, E_t\cap E_{4t/\kappa,R},G) \le C\, 
\mathrm{cap}_{s,p,\omega}(E,E_{t},G)\,,
\end{equation}
where $C=\kappa^{-np}\,C(n,p,[\omega]_{A_p})$.
\end{theorem}

Before proving this result, let us illustrate the special case of relative $(s,p)$-capacity 
while working in the setting of Theorem \ref{t.cap_comp} together with
a fixed $0<\alpha<1$.
Since $E_{t^{1/\alpha}}\subset E_t$ for all small $t>0$, we have 
\begin{equation}\label{e.dir}
\mathrm{cap}_{s,p}(E,E_t,G)\le \mathrm{cap}_{s,p}(E,E_{t^{1/\alpha}},G)\,.
\end{equation}
Because the set $E_{t^{1/\alpha}}$ can be much smaller
than $E_t$, it is reasonable to expect---unless
both of the above capacities vanish---that the converse of inequality \eqref{e.dir}
with a
$t$-independent constant
cannot hold
for all small $t>0$.
For a more precise statement, we need the following non-trivial example.

\begin{example}\label{e.conv_fail}
Let $E\subset G$ be a compact Ahlfors $\lambda$-regular set \cite{L}
with $0<\lambda <n$; then $E$ is $\kappa$-porous for some $0<\kappa<1$.
Assume that $0<s<1$ and $1< p < \infty$ satisfy  $n-sp<\lambda < n$. 
Then there exists $t_0>0$ such that, whenever $0<t<t_0$, we have
\begin{equation}\label{e.ahlfors}
\mathrm{cap}_{s,p}(E,E_t,G)\simeq t^{n-\lambda -sp}
\end{equation}
and the constants of comparison are independent of $t$.
Indeed, this comparison estimate can be obtained by adapting 
the arguments that are given in \cite{L}; we omit the details here.
\end{example}

Let us continue our discussion (before the example) and suppose that  
$\alpha s<\sigma<1$. If we take $n-\alpha sp<\lambda<n$ to be sufficiently large, then Example \ref{e.conv_fail} shows that inequality
\begin{equation}\label{e.cap_sharp}
\mathrm{cap}_{\sigma,p}(E,E_{c\,t^{1/\alpha}},G) \le C\, \mathrm{cap}_{s,p}(E,E_t,G)\,,\qquad t>0\text{ small}\,,
\end{equation}
fails for some compact $\kappa$-porous set $E\subset G$ if
$c$ and $C$ are  not allowed to depend
on the parameter $t>0$ (but are allowed to depend on $E$ and the other parameters).

On the other hand, if we assume that $\sigma \le \alpha s$,
then inequality \eqref{e.cap_sharp} holds for any fixed compact Ahlfors $\lambda$-regular
set $E\subset G$ given that $n-\sigma p<\lambda <n$; see Example \ref{e.conv_fail}. 
If $\sigma < \alpha s$ (we now exclude the `critical' case  $\sigma=\alpha s$), then the
last conclusion can be independently obtained with our results:   
by straightforward estimates and Theorem \ref{t.cap_comp} we find that, for small $t>0$,
\begin{align*}
\mathrm{cap}_{\sigma,p}(E,E_{c\,t^{1/\alpha}},G)&\lesssim \mathrm{cap}_{s+\varepsilon/p,p,\omega,R}(E,E_t\cap E_{4t/\kappa,R},G)
\\&\lesssim \mathrm{cap}_{s+\varepsilon/p,p,\omega}(E,E_t,G) = \mathrm{cap}_{s,p}(E,E_t,G)\,.
\end{align*}
Here $c=(4/\kappa)^{1/\alpha}$, 
$\omega = \lvert \cdot\rvert^{\varepsilon-n}$ (for a sufficiently small $\varepsilon>0$)
is an $A_p$ weight,
and 
\[R(x)=\min\{ \mathrm{dist}(x,E)^\alpha,\mathrm{dist}(x,\partial G)\}\,,\qquad x\in G\,,
\] defines an $\alpha$-H\"older function on the bounded open set $G$.
We also have  $R(x)=\mathrm{dist}(x,E)^\alpha$ if $x$ is 
sufficiently close to a fixed $\kappa$-porous compact set $E\subset G$. In particular,
this set is allowed to be a compact Ahlfors $\lambda$-regular set with $0<\lambda<n$.
We turn our focus to the proof of  Theorem \ref{t.cap_comp}. To this end, we first
consider the following 
modification of \cite[Lemma 2.3]{Dyda3}.

\begin{lemma}\label{l.continuity}
Suppose that $R:G\to \R$ is a
continuous function on an  open set $\emptyset\not=G\subsetneq\R^n$ such that $0\le R(x)\le \mathrm{dist}(x,\partial G)$ for every
$x\in G$. Assume that $f:G\to \R$ has a continuous
extension to $\R^n$. 
Then $M_R(f)=M_R(f\chi_G)$ is a continuous function on $G$.
\end{lemma}

\begin{proof}
We first observe that the function
defined by $F(x,0)=\lvert f(x)\rvert$ and \[F(x,r)=\intav_{B(x,r)} \lvert f(y)\rvert\,dy\]
 for $r>0$
 is continuous on $\R^n\times [0,\infty)$
(in this definition the function $\lvert f\rvert$ is continuously extended to the
whole $\R^n$, which is possible due to assumptions). 

Let us fix $x\in G$ and $\varepsilon>0$.
By the uniform continuity of $F$
on $\overline{B}(x,1) \times [0,R(x)+1]$,
there exists $0<\eta<1$ such that
$\lvert F(y,s)- F(x,t)\rvert < \varepsilon$
whenever $\lvert y-x\rvert+\lvert s-t\rvert<\eta$ and $0\leq s,t \leq R(x)+1$.
Moreover, by continuity of $R$ at $x$, there exists $0<\delta<\eta/2\wedge \dist(x,\partial G)$ such that
\[
\lvert R(x)-R(y)\rvert < \frac{\eta}{2}
\]
whenever $\lvert x-y\rvert < \delta$. 
To prove the continuity of $M_R(f)$ at the point $x$, let us consider a point $y\in G$ such that $\lvert x-y\rvert < \delta$. 
Now, for some $0\le r(y)\le R(y)$, we have
\[
M_R (f)(y) = F(y,r(y)) \leq
   F(x,r(y) \wedge R(x)) + \varepsilon \leq M_R(f)(x)+\varepsilon\,,
\]
because $\lvert y-x\rvert + \lvert r(y) - r(y) \wedge R(x)\rvert < \eta$
and $0\le r(y)\le R(y)\le R(x)+1$.
On the other hand, for some $0\le r(x)\le R(x)$,
\[
 M_R(f)(x) = F(x,r(x))  \leq
 F(y,r(x) \wedge R(y)) + \varepsilon
\leq  M_R(f)(y)+ \varepsilon\,.
\]
This proves continuity of $M_R(f)$.
\end{proof}

\begin{proof}[Proof of Theorem \ref{t.cap_comp}]
Fix $t>0$ as in the statement of the theorem.
Let $\varphi\in C_0(G)$ 
be such that $\varphi\ge 1$ on $E$ and $\varphi=0$ on $G\setminus E_t$. We define
\[
f=1-\min\{1,\varphi\} \in C(G)\,.
\]
Then 
$f(x)=0$ for every $x\in E$ and 
$f(x)=1$ if $x\in G\setminus E_t=\{x\in G\,:\, \mathrm{dist}(x,E)\ge t\}$.
Let us consider a point $x\in G$ which satisfies 
$R(x)\ge 4t/\kappa$. By using the $\kappa$-porosity of $E$, it is rather straightforward
to find $y\in B(x, 4t/\kappa)$ such
that $B(y, t)\subset B(x,4t/\kappa)\subset G$ and
$\mathrm{dist}\big(  B(y, t), E\big) \ge t$. Hence, 
\begin{align*}
M_R (f)(x) &\ge \intav_{B(x,4t/\kappa)} \lvert f(z)\rvert\,dz
\ge \frac{\lvert B(y,t)\rvert}{\lvert B(x,4t/\kappa)\rvert} \intav_{B(y,t)} \lvert f(z)\rvert\,dz
\ge \frac{\lvert B(y,t)\rvert}{\lvert B(x,4t/\kappa)\rvert} =4^{-n}\kappa^{n}\,.
\end{align*}
It follows that $4^n\kappa^{-n}M_R(f)\ge 1$ on the set $G\setminus
E_{4t/\kappa,R}$. 
We also have $M_R(f)\ge \lvert f\rvert= 1$ on the set $G\setminus  E_{t}$.
Moreover, if $x\in E$, then $R(x)=0$ and hence
\[M_R(f)(x)=\lvert f(x)\rvert=0\,.\]
Since $f:G\to\R$ can be continuously extended to the whole $\R^n$ by setting $f(x)=1$
for every $x\in\R^n\setminus G$, 
by  Lemma \ref{l.continuity} we find that  $M_R(f)$ is 
continuous on $G$.

By using the previous facts, we find that the function
$g=1-\min\{1,4^{n+1}\kappa^{-n}M_R(f)\}$ is an admissible
test function for the capacity  
in the left-hand side of inequality \eqref{e.refss}; in particular,
the support condition follows from the chain of inclusions \[
\mathrm{supp}(g)
\subset \{x\in G\,:\, 4^{n+1} \kappa^{-n} M_R(f)(x)\le 1\} 
\subset E_t\cap E_{4t/\kappa,R}\subset \overline{E_t}\subset G\,.\]
Hence, we obtain that
\begin{align*}
&\mathrm{cap}_{s,p,\omega,R} ( E, E_t\cap E_{4t/\kappa,R},G)
\le \int_G \int_G \frac{\lvert g(x)-g(y)\rvert^p}{(\lvert x-y\rvert + \lvert R(x)-R(y)\rvert)^{sp}}\,\omega(x-y)\,dy\,dx\\
&\le (4^{n+1}\kappa^{-n})^p \int_G \int_G \frac{\lvert M_R(f)(x)-M_R(f)(y)\rvert^p}{(\lvert x-y\rvert + \lvert R(x)-R(y)\rvert)^{sp}}\,\omega(x-y)\,dy\,dx\,.
\end{align*}
Observe that $f\in L^\infty(G)\subset L^p(G)$
since $G$ is assumed to be bounded. By Theorem \ref{t.m.bounded}, 
we then obtain that
\begin{align*}
\mathrm{cap}_{s,p,\omega,R} ( E, E_t\cap E_{4t/\kappa,R},G)&\le \kappa^{-np} 
C({n,p,[\omega]_{A_p}}) \int_G\int_G\frac{\lvert f(x)-f(y)\rvert^p}{\lvert x-y\rvert^{sp}}\,\omega(x-y)\,dy\,dx\\
&\le \kappa^{-np}C(n,p,[\omega]_{A_p}) \int_G\int_G\frac{\lvert \varphi(x)-\varphi(y)\rvert^p}{\lvert x-y\rvert^{sp}}\,\omega(x-y)\,dy\,dx\,.
\end{align*}
The required inequality \eqref{e.refss} follows by taking infimum over all of the functions $\varphi$ that are considered above.
\end{proof}

\end{document}